\documentclass[11pt]{amsart}
\usepackage{amssymb,amsmath,amsfonts,amscd,euscript}
\usepackage[all]{xy}
\usepackage{color}

\usepackage[backref=page]{hyperref}
\usepackage{todonotes}

\usepackage{amsthm,enumitem,bm,pict2e,xcolor}
\usepackage{indentfirst}

\newcommand{\nc}{\newcommand}

\numberwithin{equation}{section}
\newtheorem{thm}{Theorem}[section]
\newtheorem{prop}[thm]{Proposition}
\newtheorem{lem}[thm]{Lemma}
\newtheorem{cor}[thm]{Corollary}
\newtheorem{rem}[thm]{Remark}

\newtheorem{example}[thm]{Example}
\newtheorem{dfn}[thm]{Definition}

\newtheorem{conj}[thm]{Conjecture}

\nc{\gl}{\mathfrak{gl}}
\nc{\GL}{\mathfrak{GL}}
\nc{\g}{\mathfrak{g}}
\nc{\gh}{\widehat\g}
\nc{\fh}{\mathfrak{h}}
\nc{\la}{\lambda}
\nc{\al}{\alpha }
\nc{\be}{\beta }
\nc{\ve}{\varepsilon }
\nc{\om}{\omega }

\nc{\ta}{\theta}
\nc{\veps}{\varepsilon}
\nc{\ch}{{\mathop {\rm ch}}}
\nc{\Tr}{{\mathop {\rm Tr}\,}}
\nc{\Id}{{\mathop {\rm Id}}}
\nc{\bra}{\langle}
\nc{\ket}{\rangle}
\nc{\pa}{\partial}
\nc{\ld}{\ldots}
\nc{\cd}{\cdots}
\nc{\hk}{\hookrightarrow}
\nc{\T}{\otimes}
\nc{\Gr}{\mathrm{Gr}}
\nc{\Fl}{\mathrm{Fl}}
\nc{\ov}{\overline}

\nc{\msl}{\mathfrak{sl}}
\nc{\mgl}{\mathfrak{gl}}
\nc{\U}{\mathrm U}
\nc{\rot}{\mathrm{rot}}

\newcommand{\bC}{{\mathbb C}}

\newcommand{\bZ}{{\mathbb Z}}

\newcommand{\bP}{{\mathbb P}}
\newcommand{\bA}{{\mathbb A}}

\newcommand{\bO}{{\mathbb O}}
\newcommand{\bK}{{\mathbb K}}

\newcommand{\bfm}{{\bf m}}
\newcommand{\eO}{\EuScript{O}}

\newcommand{\bfJ}{{\bf J}}
\newcommand{\bfd}{{\bf d}}

\nc{\aGr}{{\mathcal Gr}}
\nc{\aFl}{{\mathcal AF_n}}
\newcommand{\bI}{{\bf I}}

\begin{document}

\title[Global positroid varieties]
{Global positroid varieties}

\author{Evgeny Feigin}
\address{Evgeny Feigin:\newline
School of Mathematical Sciences, Tel Aviv University, Tel Aviv, 69978, Israel}
\email{evgfeig@gmail.com}

\begin{abstract}
Positroid subvarieties of complex Grassmannians are the images of the 
Richardson subvarieties of the full flag varieties under the natural projection 
map. Positroid varieties admit natural embeddings into certain quiver Grassmannians
for equioriented cyclic quivers. Varying representations of the quiver, one defines global
positroid varieties inside the type A global affine Grassmannian. The general fiber 
of the family is isomorphic to the corresponding classical positroid variety and the special 
fiber is a subvariety of the juggling variety. 
We show that the global positroid families are flat and describe their (conjecturally reduced) 
scheme structures. We also describe the special fibers inside the product of classical 
positroid varieties and realize the irreducible components of the special fibers in terms 
of the affine Richardson varieties.
\end{abstract}

\maketitle

\section{Introduction}
Positroid varieties form a stratification of the classical Grassmannians $\Gr_{k,n}$.
This stratification has many beautiful geometric, algebraic and combinatorial properties 
\cite{KLS13,Lam14,Sp26}.
The positroid stratification appears naturally in the theory of totally nonnegative flag varieties 
\cite{GKL22,Lus98a,Lus98b,Riet06,W07},
combinatorics \cite{Sn10,W05,Pos06}, cluster algebras \cite{FS-B22,GL23,GS24}, representation theory of quivers \cite{FLP22,FLP23a}. 
The positroid varieties are labeled by various combinatorial objects, such as juggling patterns, 
Grassmann necklaces, bounded affine permutations, decorated permutations \cite{Lam14,Pos06,Sp26}.
In this paper we mostly use the labeling by the juggling patterns $\bfJ$ -- collections 
$\{J_b,\ b\in\bZ_n\}$ 
of cardinality $k$ subsets of the set $\{1,\dots,n\}$ such that $j-1\in J_{b+1}$ whenever 
$j\in J_b$, $j>1$.
Let us denote by $\Pi_{\bfJ}$ the positroid variety corresponding to a juggling pattern $\bfJ$.
In particular, each $\Pi_{\bfJ}\subset \Gr_{k,n}$ is obtained as the image of a Richardson variety
inside the classical flag variety $SL_n/B$ under the natural projection $SL_n/B\to\Gr_{k,n}$. 

The Grassmannians admit a remarkable flat degeneration to the juggling variety \cite{Kn08} inside the global affine Grassmannian \cite{AB24,Ga01,PR08}.
The corresponding family
over the affine line shows up in arithmetic \cite{Go01,PRS13,He13},
geometric representation theory \cite{Ga01,Zhu17,Zhou19}, and quiver theory \cite{FLP22,FLP23a}.
The special fiber of the family is a union of Schubert varieties inside the affine flag variety
\cite{HL15,Zhu14}. The simplest way to describe the family is as follows: for each $b\in\bZ_n$ 
let us fix an $n$-dimensional vector space $W^{(b)}$; we 
fix bases $w^{(b)}_i$, $i=1,\dots,n$ of the spaces $W^{(b)}$. The family $\underline{\Gr}_{k,n}$
consists of collections of the form $(\veps,U_0,\dots,U_{n-1})$ such that $\veps\in\bA^1$,
$U_b\in\Gr_k(W^{(b)})$ and $M(\veps)U_b\subset U_{b+1}$, where $M(\veps)$ is a linear map 
sending $w^{(b)}_i$ to $w^{(b+1)}_{i-1}$ for $i>1$ and sending $w^{(b)}_1$ to $\veps w^{(b+1)}_n$. 
The map $\underline{\Gr}_{k,n}\to\bA^1$
is the projection to the first coordinate; the preimage $\pi^{-1}(0)$ is the juggling variety
and all other fibers are isomorphic to $\Gr_{k,n}$. 
We note that the family above is a special case of the global Schubert varieties inside affine Grassmannians
\cite{AB24,F25,FLP23b,HY24} (one also finds similar generalizations 
on the totally nonnegative side \cite{LaPy12,LaPy13,RW24}).

The explicit description above is closely related to the positroid varieties $\Pi_\bfJ$.
Namely, for any nonzero $\veps$, $\Pi_\bfJ$ is isomorphic to the intersection of $\pi^{-1}(\veps)$
with the product of Schubert varieties $X^-_{J_b}$ in $\Gr_k(W^{(b)})$ corresponding to the
$k$-tuples $J_b$. We now give the main definition. The global positroid variety
inside the global Grassmannian is defined as
\[
\underline{\Pi}_\bfJ = \underline\Gr_{k,n}\bigcap \left( \bA^1\times \prod_{b\in\bZ_n} X^-_{J_b} \right)
\]  
equipped with the reduced scheme structure. 
Slightly abusing notation, we denote by the same letter $\pi$ the projection from the global 
positroid variety to the first factor. 
Our goal is to study the whole family $\underline{\Pi}_\bfJ$ as well as the special fiber 
$\Pi_\bfJ(0)\subset \underline{\Pi}_\bfJ$.
Using the classical Pl\"ucker embeddings for the Grassmannians $\Gr_k(W^{(b)})$, one gets
an embedding of $\underline{\Pi}_\bfJ$ into the product  
of the affine line and a collection of projective spaces, which gives the reduced 
scheme structure on the global positroid varieties and on the fibers of the map $\pi$. 

Our first main result is the following.
\begin{thm}
For any juggling pattern $\bfJ$ the global positroid family $\underline{\Pi}_\bfJ$ is flat.
The closure of $\pi^{-1}(\bA^1\setminus \{0\})$ is equal to the whole family.
\end{thm} 

In order to describe  $\Pi_\bfJ(0)$, we recall that the special fiber of the global Grassmannian
is isomorphic to a union of $\binom{n}{k}$ Schubert varieties $Z_{w_S}$ labeled by
$k$-element subsets $S\subset \bZ_n$ 
inside the affine flag variety $\aFl$ for the affine group ${GL}_n$.
Here $w_S$ are certain elements of the extended affine Weyl group $W^{ext} = S_n\ltimes \bZ^n$ (see Section \ref{sec:aGraFl} for the precise definition). 
The torus fixed points inside this union
are labeled by the juggling patterns, which are in bijection with the set of the 
extended affine Weyl 
group elements which are smaller than or equal to one of the elements $w_S$.
For a juggling pattern $\bfJ$, we denote the corresponding affine Weyl group element by $w_\bfJ$.

Let $R_u^w\subset \aFl$ be the Richardson varieties inside the affine flag variety, $u,w\in W^{ext}$,
$u\le w$ \cite{KL80,Rich92} -- the intersection of the Schubert 
variety labeled by $w$ with an opposite Schubert variety corresponding to $u$. 
Here is our second theorem.

\begin{thm}
For a juggling pattern $\bfJ$, the special positroid variety $\Pi_\bfJ(0)$ is (in general) reducible 
and equidimensional with $\dim \Pi_\bfJ(0)=\dim \Pi_\bfJ$. Each irreducible component is  
isomorphic to a Richardson variety $R_{w_\bfJ}^{w_S}$ for a $k$-element subset $S\subset \bZ_n$, 
$w_S\ge w_\bfJ$.
\end{thm} 

We also provide a defining ideal for $\Pi_\bfJ(0)$ and conjecture that the ideal is radical.
We prove the conjecture for $k=1$ in the Appendix using certain generalized rectangular 
semistandard tableaux; analogous tableaux are not yet available for $k>1$.  
More precisely, $\Pi_\bfJ(0)$ is realized as a subvariety in the product of projective spaces
and hence the defining ideal and the multihomogeneous coordinate ring are multigraded
(the Pl\"ucker variables $\Delta^{(a)}_I$ come equipped with colors $a\in\bZ_n$).
Hence it is natural 
to ask for a generalized (colored) version of the semistandard tableaux (see \cite{AGH23,Lam19,EH24}
for similar results in the uncolored case).

Finally, we express the special positroid varieties in terms of the classical ones as follows. Let 
$\mathrm{rot}$ be an automorphism of the set of juggling patterns which rotates the entries:
$\mathrm{rot}(J_0,\dots,J_{n-1})=(J_1,J_2,\dots,J_{n-1},J_0)$. We prove the following theorem.

\begin{thm}
For a juggling pattern $\bfJ$, the special positroid variety $\Pi_\bfJ(0)$ is given by
\[
\Pi_\bfJ(0) = \Gr_{k,n}(0) \bigcap \left(\prod_{b\in\bZ_n} \Pi_{\mathrm{rot}^b\bfJ}\right).
\]
\end{thm} 
 
This theorem gives a map  $\Pi_\bfJ(0)\to  \Pi_{\mathrm{rot}^b\bfJ}$. 
We conjecture that this map is surjective. More precisely, since 
$\Pi_\bfJ(0)$ is a union of the Richardson varieties $R_{w_\bfJ}^{w_S}$, it suffices to
find an $S\in \binom{\bZ_n}{k}$ such that $R_{w_\bfJ}^{w_S}$ surjects onto 
$\Pi_{\mathrm{rot}^b\bfJ}$. We show that this is implied by the following conjecture. Here we use 
the notation $u \le_{W_b} v$ for the parabolic $W_b$-Bruhat order, where $W_b\simeq S_n$ is the 
$b$-th standard maximal parabolic subgroup of $W^{ext}$. 

\begin{conj}
For any juggling pattern $\bfJ$ and $b\in \bZ_n$, there exists a unique $S\in \binom{\bZ_n}{k}$ 
such that $w_\bfJ\le_{W_b} w_S$. Equivalently, there exists a unique affine Richardson variety 
of the form $R_{w_\bfJ}^{w_S}$ which surjects onto $\Pi_{\mathrm{rot}^b\bfJ}$.
\end{conj}
 
We close the paper with the appendix, where the $k=1$ case is worked out in full detail. 
 
Our paper is organized as follows.  In Section \ref{sec:prelim} we collect preliminaries on 
Schubert varieties, quiver Grassmannians and affine flag varieties. In Section \ref{sec:global} 
we introduce the global positroid varieties and study their properties. Section \ref{sec:special}
 is devoted to the study of the special fiber of the family. In the Appendix we give all the details 
 for the  $k=1$ case. 

\section*{Acknowledgments}
We are very grateful to the anonymous referee for many valuable comments and suggestions.
EF was partially supported by the ISF grant 493/24.

\section{Preliminaries}\label{sec:prelim}
\subsection{Grassmannians, flags, Schubert varieties}
Let us fix two positive integers $k$ and $n$, $1\le k <n$; we set 
$[n]=\{1,\dots,n\}$. Let $W$ be an $n$-dimensional 
complex vector space with a basis $w_1,\dots,w_n$. Let $\Gr_{k,n}$ be the Grassmannian of $k$-dimensional subspaces in $W$. For an element $\sigma$ of the Weyl group $S_n$, we denote by $p_\sigma\in \Gr_{k,n}$
the subspace spanned by the basis elements $w_{\sigma(i)}$, $1\le i\le k$.
The group $GL_n$ acts on the Grassmannians. The points $p_\sigma$ 
are exactly the fixed points for the torus $T\subset SL_n$ of diagonal matrices. 

One has the Pl\"ucker embedding $\Gr_{k,n}\to \bP(\Lambda^kW)$;
the image of this embedding is defined by the  Pl\"ucker ideal
generated by certain quadratic expressions in Pl\"ucker variables $\Delta_I$, $I\in\binom{[n]}{k}$ (dual to the wedge products of the basis vectors $w_i$).
The expressions are labeled by collections $I,J\in\binom{[n]}{k}$ and a number $r=1,\dots,k$.
The corresponding generator is explicitly given by 
\[
g_{I,J} = \Delta_I\Delta_J - \sum \Delta_{I'}\Delta_{J'},
\]    
where the pairs ${I'},{J'}$ in the right hand side are obtained from $I,J$ by interchanging the 
element $j_r$ in $J$ with all elements of $I$ (keeping in mind that 
$\Delta_{\tau I}=(-1)^\tau \Delta_I$ for $\tau\in S_k$).

Let $B,B^-\subset SL_n$ be the Borel subgroups of upper- and lower-triangular matrices 
(in particular, both $B$ and $B^-$ contain $T$).  
For a Weyl group element $\sigma\in S_n$ let 
$\mathring X_\sigma = B.p_\sigma\subset \Gr_{k,n}$ and $\mathring X^-_\sigma=B^-.p_\sigma\subset \Gr_{k,n}$ be the open Schubert varieties -- the orbits  
of the Borel subgroup $B$ and the opposite Borel subgroup $B^-$. 
The closed Schubert varieties are the closures of the corresponding orbits:
$X_\sigma = \overline{\mathring X_\sigma}$, 
$X^-_\sigma = \overline{\mathring X^-_\sigma}$. 
In particular, for the longest element $w_0\in S_n$, one has 
$X_{w_0}=X^-_{e}=\Gr_{k,n}$ and $X_e=p_e$, $X^-_{w_0}=p_{w_0}$ are just points. 

Schubert varieties $\mathring X_\sigma$, $X_\sigma$, $\mathring X^-_\sigma$, $X^-_\sigma$
inside $\Gr_{k,n}$ depend only on the image $\sigma[k]$, so we sometimes denote 
$\mathring X_\sigma$ by $\mathring X_{\sigma[k]}$ (and similarly for other Schubert varieties).
For two cardinality $k$ subsets $I,J\subset [n]$, $I=(i_1<\dots<i_k)$, $J=(j_1<\dots <j_k)$,  
we write $I\le J$ if $i_u\le j_u$ for all $u\in [k]$. We note that $\Delta_I$ vanishes on 
$X_\sigma$ unless $I\le \sigma([k])$. Similarly, $\Delta_I$ vanishes on $X^-_\sigma$ unless $I\ge \sigma([k])$. 

Let $\Fl_n=SL_n/B$ be the variety of complete flags consisting of
collections $(V_i)_{i=1}^{n-1}$ of subspaces of the space $W$ such that $\dim V_i=i$
and $V_i\subset V_{i+1}$. In particular, one has an
obvious embedding $\Fl_n\subset \prod_{k=1}^{n-1} \Gr_{k,n}$. 
Slightly abusing notation, we denote by the same symbol $p_\sigma\in\Fl_n$ the product of the corresponding points in the Grassmannians. The open Schubert varieties and 
their opposite analogues in
the flag varieties are defined (as in the Grassmann case) by
\[
\mathring Y_\sigma = B.p_\sigma\subset \Fl_n,\quad 
\mathring Y^-_\sigma = B^-.p_\sigma\subset \Fl_n;
\] 
the closed Schubert varieties are the orbit closures.

Finally, let us define the Richardson varieties inside $\Fl_n$. 
Given two Weyl group elements
$\sigma,\tau\in S_n$ the open Richardson variety 
$\mathring R_{\sigma}^\tau\subset \Fl_n$ is defined 
as the intersection $\mathring Y^-_\sigma\cap \mathring Y_\tau$. 
One defines the closed Richardson varieties $R_\sigma^\tau$ as the closure of  
$\mathring R_\sigma^\tau$.
Clearly, $R_\sigma^\tau=\emptyset$ unless $\sigma$ is smaller than or equal 
to $\tau$ in the Bruhat
order in $S_n$ (if $\sigma=\tau$, then $R_\sigma^\tau$ is a single point $p_\sigma$). 

\begin{rem}
The definition of the Schubert varieties, opposite Schubert varieties and Richardson varieties 
is one and the same in all Lie types and for all partial flag varieties. However, in this
paper we only treat the case of finite and affine types $A$; in both cases we will only need
Grassmannians and complete flag varieties. 	
\end{rem}

\subsection{Juggling patterns, Grassmann necklaces, and affine permutations}
Juggling patterns, Grassmann necklaces, and affine permutations are combinatorial objects 
parametrizing cells in the totally nonnegative Grassmannians and positroid varieties 
(see \cite{Pos06,Lam14,KLS13}). Here we recall basic definitions.   

A juggling pattern is a collection $\bfJ=(J_0,\dots,J_{n-1})$ with each $J_b$
being a cardinality $k$ subset of $[n]$ subject
to the conditions $J_b-1\subset J_{b+1}$. Here $J_b-1=\{j-1:\ j\in J_b, j>1\}$ and the lower index 
$b+1$ is taken modulo $n$ (i.e. in $\bZ_n$). We denote the set of juggling patterns by
$\mathrm{Jug}_{k,n}$. 
The set $\mathrm{Jug}_{k,n}$ is endowed with a poset structure. Namely, for two elements
$\bfJ,\bfJ'\in \mathrm{Jug}_{k,n}$,
$\bfJ=(J_b)_{b\in\bZ_n}$, $\bfJ'=(J'_b)_{b\in\bZ_n}$ 
we say that $\bfJ\le \bfJ'$ if $J_b\ge J'_b$ for all $b$. We use the standard notation:
for two sets $A=(a_1<\dots <a_k)$ and $B=(b_1<\dots <b_k)$,  we say $A\le B$ if and only 
if $a_i\le b_i$ for all $i$.

The definition of a Grassmann necklace is very close to that of a juggling pattern. 
Namely, a collection $(I_0,\dots,I_{n-1})\in \binom{[n]}{k}^{\bZ_n}$ is called 
a Grassmann necklace if $I_b\setminus \{b+1\}\subset I_{b+1}$ for all $b\in\bZ_n$ 
(the conditions include the case $b=n-1$). One has an obvious bijection between juggling
patterns and Grassmann necklaces: given $\bfJ\in \mathrm{Jug}_{k,n}$, we define 
$I_b = \{j+b,\ j\in J_b\}$, where $j+b$ is understood as $j+b-n$ whenever $j+b>n$.
One easily sees the collection $(I_b)_{b\in\bZ_n}$ thus constructed is indeed a Grassmann
necklace. The inverse map is defined in an obvious way.

Finally, let us discuss the third combinatorial structure. The group of affine permutations 
consists of bijections $w:\bZ\to\bZ$ such that  $w(i+n)=w(i)+n$. Clearly, each such map $w$ 
is defined by the values of the elements $i\in [n]$; one has $w(i) = \sigma(i) + n\la_i$ for
a permutation $\sigma\in S_n$ and a collection $\la=(\la_1,\dots,\la_n)\in \bZ^n$.
This identifies the group of affine permutations with the semidirect product 
$S_n\ltimes \bZ^n$, which in turn is isomorphic to the extended affine group $W^{ext}$
of type $A$. 
   
A bounded affine permutation of type $(k,n)$ is an element $w\in W^{ext}$ such that $i\le w(i)\le i+n$ and 
$\sum_{i=1}^n (w(i)-i) = nk$; these permutations are in bijection with $\mathrm{Jug}_{k,n}$.
Namely, given $\bfJ=(J_0,\dots,J_{n-1})\in \mathrm{Jug}_{k,n}$, the corresponding bounded 
affine permutation 
$w_\bfJ\in W^{ext}$ is defined as follows. If $J_i = \{j-1,\ j\in J_{i-1}\}$, then $w(i)=i$.
Otherwise, $J_i$ contains one element $x$ which is not of the form $j-1$ for $j\in J_{i-1}$;
we set $w(i)= i+x$. The group $W^{ext}$ is equipped with a partial Bruhat order; the map
$\bfJ\to w_\bfJ$ is order reversing (with respect to the poset structure on 
$\mathrm{Jug}_{k,n}$ described above). 

\begin{example}\label{ex:max}
The poset $\mathrm{Jug}_{k,n}$ contains a unique maximal element $\bfJ$ such that 
$J_b=[k]$ for all $b$. The corresponding element $w_\bfJ$ is given by $w(i)=i+k$ for all $i$,
which is the minimal element in the poset of bounded affine permutations of type $(k,n)$. 
\end{example}
 
\begin{example}\label{ex:minbin}
The poset $\mathrm{Jug}_{k,n}$ contains $\binom{n}{k}$ minimal elements labeled by 
$S\subset \binom{\bZ_n}{k}$: the juggling pattern $J(S)$ is defined by the condition $n\in J(S)_b$
for $b\in S$.  The corresponding bounded affine permutations $w_{\bfJ(S)}$ are given by $w(i)=i+n$ for all $i\in S$
and $w(i)=i$ otherwise. These are the maximal elements in the poset of bounded affine 
permutations of 
type $(k,n)$. By definition, $w_{\bfJ(S)}$ as an element of the semidirect product $S_n\ltimes \bZ^n\simeq W^{ext}$ is given by $\sigma = \mathrm{id}$ and $\la = (\la_1,\dots,\la_n)$ with 
$\la_i=1$ for $i\in S$ and $\la_i=0$ otherwise. 
\end{example}

\subsection{Positroid varieties}
Positroid varieties inside the Grassmannian $\Gr_{k,n}$ are defined as images of the Richardson 
varieties in the full flag  variety $\Fl_n$ with respect to the natural surjection 
$\Fl_n\to\Gr_{k,n}$ \cite{KLS13}. 
The strata of $\Gr_{k,n}$ 
of the positroid stratification are labeled by the juggling patterns, i.e. 
$\Gr_{k,n}=\bigsqcup \mathring \Pi_\bfJ$. The closed positroid varieties $\Pi_\bfJ$ are defined as
closures of $\mathring \Pi_\bfJ$. The positroid varieties $\mathring\Pi_\bfJ$
 admit the following explicit description:
 \begin{equation}\label{eq:intSch}
 \mathring\Pi_\bfJ = \bigcap_{b=0}^{n-1} \varphi^{-b}\left(\mathring X^-_{J_b}\right),\ 
 \varphi: W\to W,\ \varphi(w_i) = \begin{cases}w_{i-1}, & i>1,\\ w_n, & i=1\end{cases}  	
 \end{equation}
 and similarly for the closed positroid varieties $\Pi_\bfJ$ (with $\mathring X^-_{J_b}$ replaced
 with $X^-_{J_b}$). The closure order of the positroid varieties is compatible with the poset order on the set of juggling patterns.
 
 Each positroid variety is irreducible. The codimension of $\Pi_\bfJ$ 
 in the Grassmann variety $\Gr_{k,n}$
 is equal to the length of the corresponding bounded affine permutation: $\dim \Pi_\bfJ = k(n-k) - \ell(w_\bfJ)$,
 
\begin{example}\label{ex:full}
Let us consider the juggling pattern $\bfJ$ such that $J_b=[k]$ for 
all $b\in\bZ_n$, which is the maximal juggling pattern (see Example \ref{ex:max}). Then $\Pi_{\bfJ}=\Gr_{k,n}$ and $\mathring\Pi_{\bfJ}$ is the subset 
consisting of subspaces $U$ such that $\Delta_I(U)\ne 0$ for all $I$ of the form
$I=\{i,i+1,\dots,i+k-1\}$, $i\in [n]$ (the indices $i+j$ are understood as $i+j-n$ 
whenever $i+j>n$). 	
\end{example} 
 	
\begin{example}\label{ex:zero-dim}
Let us fix a subset $S\subset [n]$, $|S|=k$ and consider the minimal juggling 
patterns $\bfJ(S)=\{J_b(S)\}_{b\in\bZ_n}$ (see Example \ref{ex:minbin}).
Then $\mathring\Pi_{\bfJ(S)}$ is a single point. 
In fact, $w_n\in X^-_{J_b}$ for $b+1\in S$. 
Hence, $\mathring\Pi_{\bfJ(S)}=\mathrm{span}\{w_i,\ i~+~1\in S\}$ (in particular, 
if $1\in S$, then $w_n\in \mathring\Pi_{\bfJ(S)}$). These are all the zero-dimensional 
positroid varieties.
\end{example}

\subsection{Quiver Grassmannians}\label{sec:qG}   
A quiver Grassmannian $\Gr_\bfd(M)$ \cite{CI20}  is defined by a quiver $Q$, its representation $M$ and a 
dimension vector $\bfd\in \bZ_{\ge 0}^{Q_0}$, where $Q_0$ is the set of vertices of $Q$.
The points of $\Gr_\bfd(M)$ are the $\bfd$-dimensional $Q$-subrepresentations of $M$. We will only work with the cyclic equioriented quiver and a very special class of representations, so we skip the general formalism and define the concrete objects we are working with.

Let $Q$ be the equioriented cyclic quiver with $n$ vertices labeled by the elements of 
$\bZ_n$. We define
a family of $Q$-modules $M(\veps)$, depending on a parameter $\veps\in\bC$. For each vertex
$b\in Q_0$ (and for all values of the parameter $\veps$) the space $M^{(b)}(\veps)$ coincides 
with $W$. The maps $M_{b\to b+1}(\veps)$ are given as follows:
\begin{equation}\label{eq:themap}
M_{b\to b+1}(\veps): w_i \mapsto \begin{cases} w_{i-1}, & i>1\\ \veps w_n, & i=1\end{cases} 
\end{equation}
In particular, $M_{b\to b+1}(1)=\varphi$ (see \eqref{eq:intSch}).
In what follows to avoid confusion we denote the vector $w_i$ inside 
$W\simeq M^{(b)}(\veps)$ by $w_i^{(b)}$. 
A representation $M(\veps)$ for $\veps\ne 0$ decomposes as a direct sum
of $n$ copies of a representation of dimension $(1,\dots,1)$ such that the composition of all the maps around the circle is equal to $\veps\ \! \mathrm{Id}$:
\[
M(\veps) = \bigoplus_{j=1}^n \mathrm{span}\{w^{(b)}_{j-b},\ b\in\bZ_n\}
\]
(here $j-b$ is understood as $j-b+n$ for $j-b\le 0$).
If $\veps=0$, then the same formula provides a decomposition of $M(0)$ into the direct sum of
$n$ pairwise nonisomorphic nilpotent indecomposable representations of dimension $(1,\dots,1)$. 

\begin{rem}\label{rem:Grneck}
We note that all the maps in \eqref{eq:themap} are given by the same formula; in particular, for
$\veps=0$ the kernels of all the maps are spanned by $w^{(b)}_1$. For $\veps\ne 0$, the maps 
\eqref{eq:themap} are isomorphisms, but the identification of the spaces at different vertices 
shifts the lower indices of the basis vectors $w^{(b)}_\bullet$. 
This disadvantage can be fixed by shifting
the numbering of the vectors $w^{(b)}_i$; the shift has to depend on a vertex $b$. 
The price to pay is that 
after such shifts the maps  $M_{b\to b+1}(\veps)$ will depend on $b$. The isomorphism between the two
pictures is a lift of the combinatorial identification of juggling patterns and Grassmann necklaces
(see \cite{Pos06,Sp26}). 
\end{rem}

 Now let us fix a dimension vector $\bfd=(k,\dots,k)$ and denote by $\Gr_{k,n}(\veps)$ the quiver
 Grassmannian $\Gr_{(k,\dots,k)}(M(\veps))$.
 In particular, for $\veps\ne 0$ one has an isomorphism 
$\Gr_{(k,\dots,k)}(M(\veps))\simeq \Gr_{k,n}$ and 
$\Gr_{(k,\dots,k)}(M(0))$ is the so-called juggling variety, which shows up in various contexts, see \cite{Go01,Go03,Pap18,FLP23a}. 
 We also denote by $\underline\Gr_{k,n}\subset \bA^1\times\prod_{b\in\bZ_n}\Gr_{k,n}$ the global
 (universal) quiver Grassmannian consisting of collections
 $(\veps,U_0,\dots,U_{n-1})$ such that $(U_b)_{b\in \bZ_n}\in \Gr_{k,n}(\veps)$.
 We summarize the picture in the following lemma \cite{FLP22,Go01}.
 
\begin{lem}
The family $\underline\Gr_{k,n}$ is flat over $\bA^1$. 
For $\veps\ne 0$ each fiber $\Gr_{k,n}(\veps)$ is isomorphic to $\Gr_{k,n}$. The special fiber
$\Gr_{k,n}(0)$ is the juggling variety;
it is equidimensional of dimension $k(n-k)$ and has $\binom{n}{k}$ irreducible components.  
\end{lem}

\subsection{Affine Grassmannians and flag varieties}\label{sec:aGraFl}
The affine flag variety $\aFl$ of type $GL_n$ \cite{Kum02} admits a decomposition into the 
Schubert cells $\mathring Z_\tau$,
where $\tau$ is an element of the extended affine Weyl group $W^{ext} = S_n\ltimes \bZ^n$ -- 
the semidirect product of the symmetric group $S_n$ and $\bZ^n$ 
(see \cite{BB05}). For  $\sigma\in S_n$ and $\lambda\in\bZ^n$, we denote the corresponding
element of $W^{ext}$ by $\sigma z^\lambda$.
Each cell $\mathring Z_\tau$ is an orbit of the Iwahori subgroup
$\bI$ through the torus fixed point $p_\tau\in \aFl$. The affine Schubert varieties $Z_\tau$ are the 
closures of the Schubert cells. One has the dimension formula $\dim Z_\tau = \ell(\tau)$.

Similar to the finite type $A$ case, one defines the affine Richardson varieties $R_u^w$ for any 
two elements $u,w\in W^{ext}$, $u\le w$.  
The Richardson variety $R_u^w\subset \aFl$ is the intersection of $Z_w$ with the closure of the opposite Iwahori group orbit $\bI^- p_u$. 
One has the dimension formula $\dim R_u^w = \ell(w) - \ell(u)$. 

The affine flag variety $\aFl$ admits an embedding $\aFl\subset\prod_{b\in\bZ_n} \aGr_n^{(b)}$, where
$\aGr_n^{(b)}$ are the connected components of the affine Grassmannian of type $GL_n$ (see subsection \ref{sec:gaG} for more details). 
Each component $\aGr_n^{(b)}$ is identified with the quotient of the affine $GL_n$ by
the corresponding maximal parahoric subgroup. In particular, $\aGr_n^{(b)}$ is acted upon by the 
Iwahori group $\bI$; the orbits of this group are labeled by the minimal length representatives 
of the quotient $W^{ext}/W_b$, where $W_b$ is a finite Weyl group generated by the 
simple reflections $s_i$, $i\in\bZ_n$  of $W^{ext}$ such that $i\ne b$. 
We denote the torus fixed point contained in the orbit of $u\in W^{ext}/W_b$ by $p_u^{(b)}$.
The closures of these orbits are the Schubert varieties $Z_\tau^{(b)}\subset \aGr_n^{(b)}$;
hence $Z_\tau^{(b)}$ is the closure of $\bI p_\tau^{(b)}$.
The dimension formula reads as  $\dim Z_\tau^{(b)} =\ell(\tau)$ (note that $\tau$ is the 
minimal length representative in its class). 

As in the finite type case, one defines the Richardson varieties $R_u^{w;(b)}\subset \aGr_n^{(b)}$,
where $u,w$ are the minimal length representatives of $W^{ext}/W_b$. One has
$\dim R_u^{w;(b)} = \ell(w) -\ell(u)$.

The juggling variety $\Gr_{k,n}(0)$ inside the affine flag variety $\aFl$ is equal
to the union of affine Schubert varieties \cite{Go01,HL15}. 
Let us consider the subset of $W^{ext}$ consisting of elements $z^{\sigma\omega_k}$
for all $\sigma\in S_n$, where $\omega_k=(\underbrace{1,\dots,1}_k,0,\dots,0)$ is the 
$k$-th fundamental weight.   
Clearly, the elements are labeled by the cardinality-$k$ subsets $S\subset \bZ_n$:
the affine permutation labeled by $S$ is equal to $w_{\bfJ(S)}$ from Example \ref{ex:minbin}.
For brevity, we denote $w_{\bfJ(S)}$ by $w(S)$.

Hence, the juggling variety $\Gr_{k,n}(0)$ has $\binom{n}{k}$ irreducible components, 
each of dimension 
$k(n-k)$. As an immediate corollary one gets a cellular decomposition of the juggling variety 
$\Gr_{k,n}(0)$ into the union of open Schubert cells $\mathring Z_\tau$ such that there exists 
a subset $S\in \binom{[n]}{k}$ such that $\tau\le w_{S}$ with respect to the standard 
Bruhat order on the  extended affine Weyl group. 
The cells in the Schubert decomposition of the juggling
variety are labeled by the $(k,n)$-juggling patterns: these are the open Schubert varieties 
corresponding to the Weyl group elements $w_\bfJ$ for
$\bfJ\in \mathrm{Jug}_{k,n}$.

\subsection{Global affine Grassmannians}\label{sec:gaG}
The variety $\underline\Gr_{k,n}$ can be seen as a global Schubert variety 
inside the global affine Grassmannian \cite{AB24,Zhu14}.
The  global affine Grassmannians are defined in all types, but we will only need the type
$A$ case, where one has the lattice realization \cite{Kac85,Kum02}. Let us recall the formalism.

Let $\bK=\bC((t))$ be the field of Laurent series in a variable $t$ and let $\bO=\bC[[t]] $
be the ring of formal power series. 
We fix an $n$-dimensional vector space $V$ with a basis $v_1,\dots,v_n$.
The $GL_n$ type affine Grassmannian $\aGr_n$ consists of lattices in the space $V(\bK)=V\T\bK$, i.e.
free $\bO$ modules of rank $n$ and is identified with the quotient $GL_n(\bK)/GL_n(\bO)$. 
The affine Grassmannian $\aGr_n$ is the disjoint union of $\bZ$ connected components $\aGr^{(b)}_n$, $b\in\bZ$, which are distinguished by the value of the valuation function. The components $\aGr^{(b)}$ and $\aGr^{(b+n)}$ are naturally identified, so 
we will only consider $\aGr^{(b)}$ for $b\in\bZ_n$. ` 

A valuation $\nu(L)$ of a lattice $L$ is the minimal power of $t$ showing up in the determinant
of the matrix with columns $f_1,\dots,f_n$, where $\{f_i\}_{i=1}^n$ is an $\bO$-basis 
of $L$, i.e. $\nu(L)=\mathrm{ord}_t\det(f_1,\dots,f_n)$ (note that this number does not
depend on a choice of the $\bO$-basis $\{f_i\}_{i=1}^n$). A chain of lattices is a sequence of
 lattices
\[
L_0\subset L_1\subset \dots \subset L_{n-1}\subset L_n = t^{-1}L_0, \quad  \dim_\bC L_{b+1}/L_b=1. 
\] 
In particular, $\nu(L_{b+1})=\nu(L_b)-1$. The set of lattice chains is the affine flag variety $\aFl$
of type $GL_n$.  It is identified with the quotient $GL_n(\bK)/\bI$, where $\bI\subset GL_n(\bO)$ 
is the Iwahori group -- the preimage of the (upper-triangular) Borel subgroup $B\subset GL_n$
with respect to the $t=0$ evaluation map $GL_n(\bO)\to GL_n$. By construction, $\aFl$ admits an embedding into the product $\prod_{b\in \bZ_n} \aGr_n^{(b)}$.  

The global affine Grassmannian is a family over $\bA^1$ such the special fiber over $0$ is
isomorphic to the affine flag variety $\aFl$ and all other fibers are affine Grassmannians $\aGr_n$.
This infinite-dimensional family admits finite-dimensional subfamilies, whose general fibers are
(spherical) Schubert varieties. In this paper we only treat a very particular case of classical
Grassmannians, corresponding to minuscule weights. We formulate the construction
in this particular case.

For $b=0,\dots,n-1$, let $\theta_b:V(\bK)\to V(\bK)$ be the following  $\bK$-linear 
maps 
\[
\theta_b(v_i)=\begin{cases}
(t+\veps) v_{i}, & i=b+1,\\ v_{i}, & i\ne b+1	
\end{cases}
\]
(note that the composition of all the maps $\theta_b$ is equal to $(t+\veps)\mathrm{Id}$).
The global affine Grassmannian is formed by collections $(\veps,L_0,\dots,L_{n-1})$, where $\veps$ is a complex number, $L_b$ are lattices of equal valuation ($\nu(L_0)=\dots=\nu(L_{n-1})$) and $\theta_b L_b\subset L_{b+1}$ for all $b$ (including $b=n-1$ with $L_n=L_0$).
A general $\veps\ne 0$ fiber is isomorphic to the affine Grassmannian, since 
$t+\veps$ is invertible in $\bO$ if and only if $\veps\ne 0$. The special $\veps=0$ fiber is isomorphic to the affine 
flag variety, the identification sends  $(L_0,\dots,L_{n-1})$ to
\[
L_0\subset (\theta_0)^{-1}_{\veps=0}L_1\subset (\theta_0\theta_1)^{-1}_{\veps=0}L_2\subset\dots \subset t^{-1}L_0.
\]
  
We consider a subvariety $\underline{\Gr}_{k,n}$ in the global affine Grassmannian (to be identified with the family from the previous section) 
cut out by the following conditions:
\[
V(\bO)\subset L_b \subset t^{-1}V(\bO),\ \nu(L_b)=-k.
\]
Clearly, $\underline{\Gr}_{k,n}$ sits inside $\bA^1\times \prod_{b\in\bZ_n} \Gr_{k,n}$ (with the ambient space being $t^{-1}V(\bO)/V(\bO)$).
Since all lattices $L_b$ contain $V(\bO)$, the conditions $\theta_b L_b\subset L_{b+1}$ after replacing lattices with the corresponding points $\bar L_b\in\Gr_{k,n}$  read as $\psi_b \bar L_b\subset \bar L_{b+1}$, where 
\begin{equation}\label{eq:psib}
\psi_b(v_i)=\begin{cases} v_{i}, & i\ne b+1,\\ \veps v_i, & i=b+1.\end{cases}.
\end{equation}
Now one passes from a collection $(\bar L_b)_b$ to the collection $(U_b)_b$
by a simple base change intertwining the maps $\psi_b$ and $M_{b\to b+1}(\veps)$ \eqref{eq:themap}
(see Remark \ref{rem:Grneck}). Namely, at vertex $b=0,\dots,n-1$ one sends $v_i$ to $w_{i-b}$ 
(the nonpositive indices to be increased by $n$).

\section{Global positroid varieties}\label{sec:global}
\subsection{The family}
Recall the embedding 
$\underline\Gr_{k,n}\subset \bA^1\times \prod_{b\in\bZ_n} \Gr_{k,n}$, which leads to the 
embedding of the fibers of the global Grassmannian to the 
product of projective spaces $\prod_{b\in\bZ_n} \bP(\Lambda^k W)$.

\begin{dfn}
For a juggling pattern $\bfJ$, the global positroid variety   
$\underline\Pi_{\bfJ}$ is a subvariety of $\underline\Gr_{k,n}$ consisting of collections $(\veps,U_0,\dots,U_{n-1})$
such that $U_b\in X^-_{J_b}$ for all $b\in\bZ_n$. 
The fibers of the natural projection map
$\underline\Pi_\bfJ\to\bA^1$ will be denoted by $\Pi_{\bfJ}(\veps)$, $\veps\in\bA^1$.
\end{dfn}

One has the following simple observation.

\begin{lem}
Let $\veps\ne 0$. Then $\Pi_\bfJ(\veps)$ is isomorphic to $\Pi_{\bfJ}$.
\end{lem}
\begin{proof}
Since $\veps\ne 0$, all the maps $M_{b\to b+1}(\veps)$ \eqref{eq:themap} are isomorphisms. 
Hence a point $(\veps,U_0,\dots,U_{n-1})\in  \Pi_\bfJ(\veps)$
is completely determined by the space $U_0$. The conditions  
$U_b\in X^-_{J_b}$  translate to 
\[
U_0\in \bigcap_{b\in \bZ_n} M^{-1}_{0\to b}(\veps)X^-_{J_b}.
\]   	
For $\veps=1$ this is exactly the definition \eqref{eq:intSch} of $\Pi_{\bfJ}$.
For general nonzero $\veps$, the matrix $M^{-1}_{0\to b}(\veps)$ is obtained from 
$M^{-1}_{0\to b}(1)$ by multiplication by an invertible  diagonal matrix. Since 
all opposite Schubert varieties $X^-_{J_b}$ are stable with respect to the diagonal 
torus, we arrive at the desired isomorphism $\Pi_\bfJ(\veps)\simeq \Pi_{\bfJ}$.
\end{proof}

Let us consider the following examples (see also Examples \ref{ex:full} and \ref{ex:zero-dim}). 

\begin{example}
Let $J_b=\{1,\dots,k\}$ for all $b\in \bZ_n$. Then $\underline\Pi_{\bfJ}=\underline\Gr_{k,n}$, since $X^-_{J_b}=\Gr_{k,n}$ for all $b$.
\end{example}

\begin{example}
For a set $S\subset \bZ_n$ of cardinality $k$, consider the juggling pattern $\bfJ(S)$ from Example \ref{ex:zero-dim}. Then for any $\veps$ the variety $\Pi_{\bfJ}(\veps)$ 
is a single point. The component $U_b$ of this point is given by 
$U_b=\mathrm{span}\{w^{(b)}_{s-b},\ s\in S\}$, where $s-b$ is understood as $n-b+s$ if $s-b\le 0$ 
(i.e. the subscripts of the vectors $w_\bullet$ always lie in $\{1,\dots,n\}$). 
We note that $U_b$ does not depend on $\veps$. 
\end{example}

We postpone the study of the topological and algebro-geometric properties of the special fibers 
$\Pi_\bfJ(0)$ until Section  \ref{sec:special} and turn to the description of the defining ideals.

\subsection{Ideal of relations}

For a juggling pattern $\bfJ$, the image of the embedding 
$\Pi_\bfJ(\veps)\subset \prod_{b\in\bZ_n} \bP(\Lambda^k W)$ is described
as the zero locus of a certain multihomogeneous radical ideal 
$\mathcal{I}_\bfJ(\veps)\subset \bC[\Delta^{(b)}_I]_{b,I}$, where $\Delta^{(b)}_I$ are the
Pl\"ucker coordinates in the $b$-th factor $\bP(\Lambda^k W)$.
So $\mathcal{I}_\bfJ(\veps)$ defines the reduced scheme structure of $\Pi_\bfJ(\veps)$
and we introduce  the Pl\"ucker algebra $\emph{Pl}_\bfJ(\veps)$ defined as the quotient 
$\bC[\Delta^{(b)}_I]/\mathcal{I}_\bfJ(\veps)$. 

We also consider the global Pl\"ucker algebra $\underline{\emph{Pl}}_\bfJ$. Namely, let us extend 
the ambient ring to  $\bC[\veps,\Delta^{(b)}_I]_{b,I}$. By definition, 
$\underline{\Pi}_\bfJ$ sits inside $\bA^1\times \prod_{b\in\bZ_n} \bP(\Lambda^k W)$. We consider
the reduced ideal $\underline{\mathcal{I}}_\bfJ\subset \bC[\veps,\Delta^{(b)}_I]_{b,I}$ of 
functions vanishing on the image of $\underline{\Pi}_\bfJ$. The quotient by this ideal
is the global Pl\"ucker algebra $\underline{\emph{Pl}}_\bfJ$.

\begin{rem}\label{rem:auxvar}
In Theorem \ref{thm:flat} we prove that the special fiber $\Pi_\bfJ(0)$ belongs to the closure 
of the preimage of $\bA^1\setminus\{0\}$ of the natural
projection map $\underline{\Pi}_\bfJ\to\bA^1$. Hence, $\underline{\mathcal{I}}_\bfJ$ is obtained from ${\mathcal{I}}_\bfJ(\veps)$, $\veps\ne 0$ by regarding $\veps$ as an auxiliary variable.
\end{rem}

For $a\in\bZ_n$, let   
${\mathcal{I}}^{(a)}_{k,n}\subset \bC[\Delta^{(a)}_I]_{I}$
be the classical  Pl\"ucker ideal, defining the Pl\"ucker embedding  
$r_a: \Gr_{k,n}\subset \bP(\Lambda^kW^{(a)})$. 
Clearly, ${\mathcal{I}}^{(a)}_{k,n}\subset \mathcal{I}_\bfJ(\veps)$ for any $a\in\bZ_n$ and
any $\veps$.

For two elements $a,b\in \bZ_n$, we define the ideal  
${\mathcal{I}}^{(a,b)}_{k,n}(\veps)\subset \bC[\Delta^{(a)}_I,\Delta^{(b)}_J]_{I,J}$ as follows.
We consider the embedding $r_{a,b}: \Gr_{k,n}\to \bP(W^{(a)})\times \bP(W^{(b)})$ given by 
$V\mapsto (r_a(V),M_{a\to b}(\veps)r_a(V))$ (slightly abusing notation, we denote by 
$M_{a\to b}(\veps)$ the map $\bP(\Lambda^k W^{(a)})\to \bP(\Lambda^k W^{(b)})$ induced by 
map \eqref{eq:themap}). We define ${\mathcal{I}}^{(a,b)}_{k,n}(\veps)$ as the reduced defining
ideal for the image of $r_{a,b}$. By \cite{Ch01}, ${\mathcal{I}}^{(a,b)}_{k,n}(\veps)$ is generated
in degree two, i.e. in bi-degrees $(2,0)$, $0,2$, and $(1,1)$. Clearly, 
${\mathcal{I}}^{(a)}_{k,n}, {\mathcal{I}}^{(b)}_{k,n}\subset {\mathcal{I}}^{(a,b)}_{k,n}(\veps)$.  

\begin{rem}
By definition, the ideal ${\mathcal{I}}^{(a,b)}_{k,n}$ depends on $\veps$. One can  
consider  ${\mathcal{I}}^{(a,b)}_{k,n}$ inside the extended ambient ring $\bC[\veps,\Delta^{(a)}_I,\Delta^{(b)}_J]_{I,J}$.
\end{rem}

Now, let ${\mathcal{I}}^{(a,b)}_{k,n}(0)\subset \bC[\Delta^{(a)}_I,\Delta^{(b)}_J]_{I,J}$ 
be the ideal consisting of all $\veps=0$ evaluations of the elements of   
${\mathcal{I}}^{(a,b)}_{k,n}(\veps)$, $\veps\ne 0$, or, equivalently, of all minimal degree terms in $\veps$
(if one considers $\veps$ as a formal variable).

\begin{example}
Let $k=2$, $n=4$. Then 	${\mathcal{I}}^{(a)}_{k,n}$ is generated by a single relation
$\Delta^{(a)}_{12}\Delta^{(a)}_{34} - \Delta^{(a)}_{13}\Delta^{(a)}_{24} + \Delta^{(a)}_{14}\Delta^{(a)}_{23}$. 

The ideal 
${\mathcal{I}}^{(a,a+1)}_{2,4}(\veps)$, $\veps\ne 0$ contains a polynomial
$\veps \Delta^{(a)}_{13}\Delta^{(a+1)}_{13} - \Delta^{(a)}_{24}\Delta^{(a+1)}_{24}$.
Hence, $\Delta^{(a)}_{24}\Delta^{(a+1)}_{24}\in {\mathcal{I}}^{(a,a+1)}_{2,4}(0)$.

The ideal 
${\mathcal{I}}^{(a,a+1)}_{2,4}(\veps)$, $\veps\ne 0$ contains a polynomial
$\Delta^{(a)}_{13}\Delta^{(a+1)}_{13} - \Delta^{(a)}_{12}\Delta^{(a+1)}_{23} - \Delta^{(a)}_{14}\Delta^{(a+1)}_{12}$. Hence, this polynomial belongs to 
${\mathcal{I}}^{(a,a+1)}_{2,4}(0)$ as well.

The ideal 
${\mathcal{I}}^{(a,a+2)}_{2,4}(\veps)$, $\veps\ne 0$ contains a polynomial
$\veps \Delta^{(a)}_{12}\Delta^{(a+2)}_{14} + \Delta^{(a)}_{23}\Delta^{(a+2)}_{34}$.
Hence,   $\Delta^{(a)}_{23}\Delta^{(a+2)}_{34}\in {\mathcal{I}}^{(a,a+2)}_{2,4}(0)$.
\end{example}

We put forward the following conjecture. 
Recall \cite{V25} that a saturation of a multihomogeneous ideal $\mathcal{I}$ in the polynomial
ring in variables $\Delta^{(a)}_I$ consists of elements $f$ such that $FB^m\subset I$ 
for an appropriate power of an irrelevant ideal $B$ (here $B$ is the product of 
irrelevant ideals in each group of variables $\Delta^{(a)}_I$, $a\in\bZ_n$).

\begin{conj}\label{conj:reduced}
For any $\veps$, the saturation of the ideal 
generated by 
${\mathcal{I}}^{(a,b)}_{k,n}(\veps)$ and  $\Delta^{(a)}_I$, $I\not\succeq J_a$ 
for all $a,b\in\bZ_n$ coincides with 
$\mathcal{I}_\bfJ(\veps)$. 
\end{conj}

We prove this conjecture for $k=1$ in the Appendix (see Corollary \ref{cor:conjk=1}).

\begin{rem}
Without the saturation the statement of  Conjecture \ref{conj:reduced} is false.
In particular, in the Appendix we consider the case $k=1$. We show that 
there are more Pl\"ucker variables vanishing on a positroid variety than those 
coming from the Schubert varieties condition.   	
\end{rem}

Let us show that set-theoretically Conjecture \ref{conj:reduced} holds true.

\begin{prop}
For any $\veps$, the common zero locus of the ideals 
${\mathcal{I}}^{(a,b)}_{k,n}(\veps)$ and relations $\Delta^{(a)}_I$, $I\not\succeq J_a$ for all
$a,b\in\bZ_n$ coincides with $\Pi_\bfJ(\veps)$.
\end{prop}
\begin{proof}
Let us consider the fibers $\Gr_{k,n}(\veps)$ of the global quiver Grassmannian for $\veps\ne 0$
(or, equivalently,
the case $J_b=[k]$ for all $b$). All these varieties can be realized as quiver Grassmannians for
cyclic quivers. 
According to \cite{LW19}, set-theoretically quiver Grassmannians are cut out by quadratic 
relations corresponding to all the paths in the underlying quiver. More precisely, 
given a representation $M$ of a quiver $Q$, one considers all maps $M_{a\to b}:M^{(a)}\to M^{(b)}$ 
for all pairs of vertices $a,b\in Q_0$. It is shown in \cite{LW19} that for a collection of
subspaces $U_a\subset M^{(a)}$, the conditions $M_{a\to b}U_a\subset U_b$ are guaranteed by
quadratic relations for the Pl\"ucker variables $\Delta_I^{(a)}$. Since all these quadratic
relations are captured by  ${\mathcal{I}}^{(a,b)}_{k,n}(\veps)$, we conclude that for $\veps\ne 0$
the ideal generated by ${\mathcal{I}}^{(a,b)}_{k,n}(\veps)$ and $\Delta^{(a)}_I$, $I\not\succeq J_a$ 
set-theoretically cuts $\Pi_\bfJ(\veps)$ inside the product of projective spaces. 
   
Now let us consider the case $\veps=0$. The degenerate  positroid varieties $\Pi_\bfJ(0)$ are
defined as the intersections of the special fiber $\Gr_{k,n}(0)$  with the product of opposite
Schubert varieties. The common zero locus of  $\Delta^{(a)}_I$, $I\not\succeq J_a$ forms the product 
of the corresponding Schubert varieties. Hence we are left to show that ideals ${\mathcal{I}}^{(a,b)}_{k,n}(0)$ set-theoretically cut out $\Gr_{k,n}(0)$ inside the product
of projective spaces.

The quadratic relations of \cite{LW19}, Theorem 2.2 are expressed in terms of matrix coefficients
of the maps $M_{a\to b}(0)$. Each such matrix coefficient of $M_{a\to b}(\veps)$ is equal to 
$0$, $1$, or $\veps$. Hence, each relation for $\veps=0$ is obtained from a relation for nonzero 
$\veps$ evaluating at $\veps=0$.
\end{proof}

The Pl\"ucker algebra $\underline{\emph{Pl}}_\bfJ$ (a.k.a. multihomogeneous coordinate ring of
the global positroid variety) is multigraded 
\[
\underline{\emph{Pl}}_\bfJ = \bigoplus_{\bfm\in \bZ_{\ge 0}^{\bZ_n}} \emph{Pl}_\bfJ(\bfm)
\] 
with each homogeneous component being a module over the polynomial ring in the variable $\veps$. 
In Theorem \ref{thm:flat} we show that the positroid families are flat. In the Appendix
we construct flat bases for the homogeneous components $\emph{Pl}_\bfJ(\bfm)$ of the Pl\"ucker 
algebra for $k=1$ in terms of certain $\bfJ$-admissible monomials. For an arbitrary $k$, one 
needs a colored analogue of the semistandard tableaux, which is not yet available.

\section{The special fiber}\label{sec:special}
Let $\rot: \mathrm{Jug}_{k,n}\to \mathrm{Jug}_{k,n}$ be the rotation automorphism of the
set of juggling patterns:
\[
\rot (J_0,\dots,J_{n-1}) = (J_1,J_2,\dots,J_{n-1},J_0).
\]
Formula \eqref{eq:intSch} implies  that the isomorphism $\varphi: W\to W$ 
(a.k.a. $M_\bullet(1)$) induces the isomorphism of positroid varieties $\Pi_\bfJ\to \Pi_{\rot(\bfJ)}$.
Let $\mathrm{pr}_b$ be the natural projection map from the juggling variety 
$\Gr_{k,n}(0)$ to the $b$-th Grassmannian factor $\Gr_{k,n}$.

\begin{prop}\label{prop:product}
The degenerate positroid variety $\Pi_\bfJ(0)$ admits the following presentation in terms of the 
classical positroid varieties:
\[
\Pi_\bfJ(0) = \Gr_{k,n}(0)\bigcap \prod_{b\in\bZ_n} \Pi_{\rot^b(\bfJ)}.
\]  
\end{prop}
\begin{proof}
Let ${\bf U}=(U_0,\dots,U_{n-1})$ be a point of $\Gr_{k,n}(0)$ such that 
$U_b\in  \Pi_{\rot^b(\bfJ)}$. Then by definition $U_b\in X^-_{J_b}$ (i.e. $\Delta^{(b)}_I(U_b)=0$ unless 
$I\ge J_b$) and hence ${\bf U}\in \Pi_\bfJ(0)$. So 
$\Pi_\bfJ(0) \supset \Gr_{k,n}(0)\bigcap \prod_{b\in\bZ_n} \Pi_{\rot^b(\bfJ)}$.
Let us prove the reverse inclusion.

We are to show that $\mathrm{pr}_b \left(\Pi_\bfJ(0)\right) \subset \Pi_{\rot^b(\bfJ)}$.
Recall the map $\varphi:W\to W$ \eqref{eq:intSch}. Our goal is to show that for a point 
${\bf U}=(U_0,\dots,U_{n-1})\in \Pi_\bfJ(0)$ one has $\varphi^a(U_b)\in X^-_{J_{a+b}}$
for any $a,b\in\bZ_n$.
Let $M_{b\to a+b}(\veps)$ denote the composition $M_{a+b-1\to a+b}(\veps) \dots M_{b\to b+1}(\veps)$.
Since ${\bf U}\in \Pi_\bfJ(0)$, we know that 
$$M_{b\to a+b}(0)U_b\subset U_{a+b}\subset X^-_{J_{a+b}},\ 
M_{a+b\to b}(0)U_{a+b}\subset U_b\subset X^-_{J_b}.$$
We will deduce from these conditions that $\varphi^a(U_b)\in X^-_{J_{a+b}}$.

Let us fix tuples $I_b,I_{a+b}\in\binom{[n]}{k}$ such that $U_{a+b}$ belongs to the open
Schubert cell $\mathring X^-_{I_{a+b}}$ and 
$\varphi^a U_{b}$ belongs to the open Schubert cell $\mathring X^-_{I_b}$. It suffices to show that 
$I_b\ge I_{a+b}$. In fact, 
our goal is to show that $\varphi^a U_b\in X^-_{J_{a+b}}$. Since ${\bf U}\in \Pi_{\bfJ}(0)$,
one has $U_{a+b}\in X^-_{J_{a+b}}$, which is equivalent to $I_{a+b}\ge J_{a+b}$. Now the 
inequality $I_b\ge I_{a+b}$ implies $I_{b}\ge J_{a+b}$ and hence $\varphi^a U_b\in X^-_{J_{a+b}}$.
We prove that $I_b\ge I_{a+b}$ in a separate lemma.
\end{proof}

\begin{lem}
Let $\varphi^a(U_b)\in \mathring X^-_{I_b}$, $U_{a+b}\in \mathring X^-_{I_{a+b}}$. Then 
$I_{a+b}\le I_b$.   	
\end{lem}
\begin{proof}
Recall the $n$-dimensional ambient spaces $W^{(b)}$ and $W^{(a+b)}$ (both isomorphic to 
the space $W$). We decompose
$W^{(b)}$ as $W^{(b)}_< \oplus W^{(b)}_>$ and $W^{(a+b)}$ as $W^{(a+b)}_< \oplus W^{(a+b)}_>$ 
with 
\begin{gather*}
	W^{(b)}_< = \mathrm{span}(w_1^{(b)},\dots,w^{(b)}_a),\ 
	W^{(b)}_> = \mathrm{span}(w^{(b)}_{a+1},\dots,w^{(b)}_n),\\
	W^{(a+b)}_< = \mathrm{span}(w_1^{(a+b)},\dots,w_{n-a}^{(a+b)}),\ 
	W^{(a+b)}_> = \mathrm{span}(w_{n-a+1}^{(a+b)},\dots,w_n^{(a+b)}).
\end{gather*}
The main properties of this decomposition we are going to use are as follows:
\begin{itemize}
	\item the map $M_{b\to a+b}$ identifies $W^{(b)}_>$ with $W^{(a+b)}_<$,
	\item the map $M_{a+b\to b}$ identifies $W^{(a+b)}_>$ with $W^{(b)}_<$,
	\item the restrictions of $M_{b\to a+b}$ and $\varphi^a$ to $W^{(b)}_>$ coincide,
	\item the restrictions of $M_{a+b\to b}$ and $\varphi^{n-a}$ to $W^{(a+b)}_>$ coincide.
\end{itemize}

Let $U_{b,>}=U_b\cap W^{(b)}_>$,  $U_{a+b,>}=U_{a+b}\cap W^{(a+b)}_>$ and similarly for 
$U_{b,<}$ and $U_{a+b,<}$. Since ${\bf U}\in\Gr(0)$,
one gets $\varphi^a U_{b,>}\subset U_{a+b,<}$ and $\varphi^{n-a} U_{a+b,>}\subset U_{b,<}$.
Now  let us consider the projection of $U_b$ to $U_{b,<}$ with the kernel being $U_{b,>}$.
The projection contains $\varphi^{n-a} U_{a+b,>}$; we fix a subspace $P_b\subset U_b$
such that the projection of $P_b$ to  $U_{b,<}$ does not intersect with $\varphi^{n-a} U_{a+b,>}$
and the direct sum of these two spaces coincides with the projection of $U_b$ to $U_{b,<}$.
Now we are ready to describe the space $\varphi^a(U_b)$: it is the direct sum of three spaces
\begin{equation}\label{eq:phiUb}
\varphi^a(U_b) = \varphi^a(U_{b,>})\oplus U_{a+b,>} \oplus \varphi^a(P_b).
\end{equation}
In fact, $U_{a+b,>}=\varphi^a(\varphi^{n-a} U_{a+b,>})$.

Let us have a closer look at the space $U_{a+b}$. We already know that
\[
U_{a+b} \supset \left(\varphi^a(U_{b,>})\oplus U_{a+b,>}\right) + M_{b\to a+b}(P_b).
\] 
We also know that $M_{b\to a+b}(P_b)$ is contained in $U_{a+b,<}$, but we have no control 
on the dimension of $M_{b\to a+b}(P_b)$. In particular, the dimension of the sum 
$\varphi^a(U_{b,>}) + M_{b\to a+b}(P_b)$ can be smaller than $\dim U_{a+b} - \dim U_{a+b,>}$.
So in order to get the whole space $U_{a+b}$ one needs to add a complementary space 
$P_{a+b}\subset U_{a+b}$ such that 
\begin{equation}\label{eq:Ua+b}
U_{a+b} = U_{a+b,>} \oplus (\varphi^a(U_{b,>}) + M_{b\to a+b}(P_b) + P_{a+b})
\end{equation}
(we note that by construction $W^{(a+b)}_>\cap (\varphi^a(U_{b,>}) + M_{b\to a+b}(P_b) + P_{a+b})=0$).

Our goal is to show that $I_{a+b}\le I_b$. Since $U_{a+b}\in \mathring X^-_{I_{a+b}}$,
$\varphi^a U_{b}\in \mathring X^-_{I_b}$, this inequality is equivalent to 
\begin{equation}\label{eq:dimpr}
\dim  \mathrm{pr}_t U_{a+b} \ge \dim  \mathrm{pr}_t (\varphi^a U_{b})\ \text{ for any } \
t\in [n],
\end{equation}
where $\mathrm{pr}_t: W^{(a+b)}\to W^{(a+b)}$ is defined as $\mathrm{pr}_t w_i =w_i$ for
$i\le t$ and $\mathrm{pr}_t w_i = 0$ for $i>t$. 
In particular, $\mathrm{pr}_{n-a}$ is equal to the projection to $W^{(a+b)}_{<}$ along
$W^{(a+b)}_{>}$. We also note that for any subspace $U\subset W^{(a+b)}$, one 
has $\dim \mathrm{pr}_t U\le \dim \mathrm{pr}_{t+1} U$.

Let us first prove \eqref{eq:dimpr} for $t\le n-a$.
The spaces $U_{a+b}$ and $\varphi^a(U_b)$ have a common subspace 
$\varphi^a(U_{b,>})\oplus U_{a+b,>}$. Since $M_{b\to a+b}(P_b)\subset W^{(a+b)}_<$,
for any $t\le n-a$ one has
\[
\mathrm{pr}_t  \varphi^a(U_{b,>}) = \mathrm{pr}_t  M_{b\to a+b}(U_{b,>}).
\] 
Hence (see \eqref{eq:phiUb}, \eqref{eq:Ua+b})
\begin{multline*}
\mathrm{pr}_t \left(\varphi^a(U_{b,>})\oplus U_{a+b,>} \oplus \varphi^a(P_b)\right) \subset\\
\mathrm{pr}_t \left( U_{a+b,>} \oplus (\varphi^a(U_{b,>}) + M_{b\to a+b}(P_b) + P_{a+b}) \right) 
\end{multline*}
as required.

Now let us prove \eqref{eq:dimpr} for $t>n-a$. Suppose that for some $t$
\[
\dim  \mathrm{pr}_t U_{a+b} < \dim  \mathrm{pr}_t (\varphi^a U_{b}).
\]
We claim that this implies that the same strict inequality holds for any $t'>t$.
However, for $t'=n$, $\dim  \mathrm{pr}_n U_{a+b} = \dim  \mathrm{pr}_n (\varphi^a U_{b})=k$.
So we are left to prove the last claim.

Since 
\[
W^{(a+b)}_>\cap \left(\varphi^a(U_{b,>}) + M_{b\to a+b}(P_b) + P_{a+b}\right) = 0,
\]
for any $t>n-a$ one gets
\[
\dim  \mathrm{pr}_{t+1} U_{a+b} - \dim  \mathrm{pr}_{t} U_{a+b} =
\dim  \mathrm{pr}_{t+1} U_{a+b,>} - \dim  \mathrm{pr}_{t} U_{a+b,>}.
\]
However, $U_{a+b,>}\subset \varphi^a (U_b)$, so 
\[
\dim  \mathrm{pr}_{t+1} \varphi^a(U_b) - \dim  \mathrm{pr}_{t} \varphi^a(U_b) \ge 
\dim  \mathrm{pr}_{t+1} U_{a+b,>} - \dim  \mathrm{pr}_{t} U_{a+b,>}.
\]
This concludes the proof of the lemma.
\end{proof}

\begin{thm}\label{thm:Richardson}
The special fiber $\Pi_{\bfJ}(0)$ has the following properties:
\begin{itemize}
\item $\Pi_{\bfJ}(0)$ is equidimensional and possibly reducible, $\dim \Pi_{\bfJ}(0)=\dim \Pi_{\bfJ}$;
\item the irreducible components of $\Pi_{\bfJ}(0)$ are the nonempty intersections of
$\Pi_{\bfJ}(0)$ with the irreducible components of $\Gr_{k,n}(0)$;  
\item each irreducible component of $\Pi_{\bfJ}(0)$ is isomorphic to a Richardson variety 
in the affine flag variety of the affine group ${GL}_n$. 
\end{itemize}
\end{thm}
\begin{proof}
We prove all the statements of the theorem together. The degenerate Grassmannian  
$\Gr_{k,n}(0)$ is embedded into
the affine flag variety and the image of the embedding is equal to the union of Schubert varieties
$Z_{w_S}$
for $S\in\binom{\bZ_n}{k}$ (see Section \ref{sec:gaG}). 
The juggling patterns label torus fixed points in the image (they correspond to a subset of 
elements of the extended affine Weyl group). Let us denote the point corresponding to $\bfJ$
by $p_\bfJ$ and the corresponding element from the extended affine Weyl group $W^{ext}$ 
by $w_\bfJ$. The condition  $U_b\in X^-_{J_b}$ means that 
the point ${\bf U}=(U_b)_b$ sits in the closure of the orbit of the opposite Iwahori group 
passing through $p_\bfJ$. 

In fact, one direction is clear: since $(p_\bfJ)_b\in X^-_{J_b}$,
the $b$-th component of any point of the closure of $\bI^- p_\bfJ$ also belongs to the same opposite 
Schubert variety. Now let us consider ${\bf U}\in \Pi_{\bfJ}(0)$ such that $U_b\in X^-_{J_b}$.
The degenerate juggling variety is realized inside the affine flag variety and hence 
${\bf U}$ belongs to an affine opposite Schubert cell $Z^-_w$. The conditions $U_b\in X^-_{J_b}$
imply that $w\ge w_\bfJ$ and hence ${\bf U}\in \overline{\bI^- p_\bfJ}$.  

We conclude that $\Pi_{\bfJ}(0)$ is the union of the affine Richardson varieties corresponding
to $w_S$ and $w_\bfJ$, which is nonempty if and only if 
$w_\bfJ$ is smaller than or equal to $w_S$.
This proves the third and second statements of the theorem (the length of $w_S$
is equal to $k(n-k)$ for any $S$).

Finally, let us prove that $\dim \Pi_{\bfJ}(0)=\dim \Pi_{\bfJ}$. By the argument above, 
the dimension
of the degenerate positroid variety $\Pi_{\bfJ}(0)$ is given by 
\[
\dim \Pi_{\bfJ}(0) = \ell(w_S) - \ell(w_\bfJ) = k(n-k)-\ell(w_\bfJ).
\] 
However, the right hand side is also equal to the dimension of the classical positroid variety 
$\dim \Pi_{\bfJ}$. 
\end{proof}

\begin{cor}
The number of irreducible components of $\Pi_{\bfJ}(0)$ is equal to the number 
of cardinality $k$ subsets $S\subset \bZ_n$ such that for any $b\in \bZ_n$ one has
$J_b\le J(S)_b$.
\end{cor}
\begin{proof}
According to the proof of the theorem above, we know that
the number of irreducible components of $\Pi_{\bfJ}(0)$ is equal to the number of 
sets $S\subset \bZ_n$ such that $\bfJ\le \bfJ(S)$, i.e. $J_b\le \bfJ(S)_b$ for
all $b\in\bZ_n$. 
\end{proof}

We know from  Proposition \ref{prop:product} that $\Pi_\bfJ(0)$ sits in the product of classical positroid varieties $\Pi_{\rot^b(\bfJ)}$.
In fact, we conjecture that the natural projection of  $\Pi_\bfJ(0)$ to the $b$-th factor (with respect to the embedding $\prod_{b\in\bZ_n} \Gr_{k,n}$) is equal to $\Pi_{\rot^b(\bfJ)}$. In order to formulate the underlying combinatorial statement, let us recall the notion of the parabolic Bruhat order.

Let $s_a$, $a\in \bZ_n$ be the simple reflections in the group $W^{ext}$. 
Together they generate the affine Weyl group, which is a subgroup of the group of affine permutations (a.k.a. the extended affine Weyl group).
For $b\in\bZ_n$, let $W_b\subset W^{ext}$ be the maximal parabolic subgroup generated by all $s_a$ except $s_b$. In particular, $W_b\simeq S_n$ for any $b$. 

Given a parabolic subgroup $P\subset W^{ext}$, the $P$-Bruhat order $\le_P$ is defined 
as follows: 
for $u,v\in W^{ext}$, we write $u\le_P v$ if there exists a chain $u=u_1,\dots,u_r=v$, $u_i\in W^{ext}$ such
that $u_{i+1}$ covers $u_i$ in the standard Bruhat order and $u_{i+1}P\ne u_iP$.    

\begin{conj}\label{conj:comb}
For any juggling pattern $\bfJ$ and any $b\in\bZ_n$, there exists a unique $S\in\binom{\bZ_n}{k}$ such that $w_\bfJ\le_{W_b} w_S$. 
\end{conj}

We prove the following corollary assuming the statement of this conjecture holds true.

\begin{cor}\label{cor:proj}
Assume 	Conjecture \ref{conj:comb} holds true. 
Then for any $\bfJ$, there exists a unique $S$ such that the composition
\[
R_{w_\bfJ}^{w_S} \subset \Pi_\bfJ(0) \subset \prod_{b\in\bZ_n} \Pi_{\rot^b(\bfJ)} \to 
\Pi_{\rot^b(\bfJ)}
\]
is birational. 
In particular, the map $\Pi_\bfJ(0)\to \Pi_{\rot^b(\bfJ)}$ is surjective.
\end{cor}
\begin{proof}
Recall the embedding $\aFl\subset \prod_{b\in\bZ_n} \aGr_n^{(b)}$ (see Section \ref{sec:aGraFl}). Now $\Gr_{k,n}(0)\subset \aFl$ and the projection of $\aFl$ to
$\aGr_n^{(b)}$ agrees with the projection of $\Gr_{k,n}(0)$ to $\Gr_{k,n}^{(b-k)}$
(recall  $\Gr_{k,n}(0)\subset \prod_{b\in\bZ}\Gr_{k,n}$, $\Gr_{k,n}^{(a)}$ denotes the
$a$-th Grassmannian). Here the shift $b\to b-k$ is introduced in order to match with
the definition of valuation from Section \ref{sec:gaG}.
We note that $\aGr^{(b)}$ is the quotient of the affine type $A$ group by the maximal
parahoric subgroup; the corresponding Weyl subgroup stabilizing the neutral point of
$\aGr^{(b)}$ is $W_b$.  

Now let us write $\Pi_\bfJ(0)$ as a union of $R_{w_\bfJ}^{w_S}$, $w_\bfJ\le w_S$.
The dimension of each Richardson variety $R_{w_\bfJ}^{w_S}$ is equal to 
$\ell(w_S) - \ell(w_\bfJ)$; since $\ell(w_S) = k(n-k)$, 
$\dim R_{w_\bfJ}^{w_S} = \dim \Pi_{\rot^b(\bfJ)}$, since
$$\dim \Pi_{\rot^b(\bfJ)} = \dim \Pi_{\bfJ} = k(n-k)-\ell(w_J).$$
Hence we are to show that there
exists a unique Richardson variety which projects birationally to the image.

By \cite{KLS14},  the projection of $R_{w_\bfJ}^{w_S}$ to the corresponding 
Richardson variety $R_{w_\bfJ}^{w_S; (b)}\subset \aGr^{(b)}$ is birational 
if and only if
$w_\bfJ\le_{W_b} w(S)$. Now Conjecture \ref{conj:comb} completes the proof.
\end{proof}

\begin{rem}
Conjecture \ref{conj:comb} holds true in the simplest case $\bfJ=([k],\dots,[k])$. 
As explained above, the conjecture is equivalent to the surjectivity of the map 	$\Pi_\bfJ(0)\to \Pi_{\rot^b(\bfJ)}$. Since $J_b=[k]$ for all $b$,
$\Pi_{\rot^b(\bfJ)}=\Gr_{k,n}$ and the degenerate positroid variety $\Pi_\bfJ(0)$
is equal to the degenerate Grassmannian $\Gr_{k,n}(0)$. The latter is the 
union of $\binom{n}{k}$ irreducible components labeled by $k$-subsets $S$
of the set $\bZ_n$. Explicitly, the component corresponding to $S$ is the
closure of the open Schubert cell containing the point $p_{\bfJ(S)}$. 
The point $p_{\bfJ(S)}=(U(S)_0,\dots,U(S)_{n-1})$ is fixed by the condition that
$U(S)_b$  contains the $n$-th basis vector $w_n$ for $b\in S$. The $a$-th 
projection of the $S$-th irreducible component can coincide with $\Gr_{k,n}$
if $U(S)_a =\mathrm{span}\{w_n,\dots,w_{n-k+1}\}$, i.e. $S=\{a,a-1,\dots,a-k+1\}$
(otherwise the projection sits in a positive-codimension Schubert cell). Now for 
 $S=\{a,a-1,\dots,a-k+1\}$ one easily shows that any point in the open cell of the
 Grassmannian  $\Gr_{k,n}^{(a)}$ extends to a point of $\Gr_{k,n}(0)$.   
\end{rem}

\begin{thm}\label{thm:flat}
The global positroid varieties form a  flat family.
\end{thm}
\begin{proof}
Our family $\pi:\underline\Pi_{\bfJ}\to \bA^1$ is defined over the affine line, 
all the fibers outside zero are isomorphic to $\Pi_{\bfJ}$ and the special fiber 
is reducible  with each
irreducible component of the same dimension $\dim\Pi_{\bfJ}$. Hence  
it suffices to check that the closure of the preimage 
$\pi^{-1}(\bA^1\setminus 0)$ coincides with the whole family \cite{V25}.
Let us consider the closure $\overline{\pi^{-1}(\bA^1\setminus 0)}$ as a family over $\bA^1$ (with the obvious projection map). Then each irreducible component of the special fiber has dimension $\dim\Pi_{\bfJ}$. In fact, 
the dimension cannot be larger (since $\dim \Pi_{\bfJ}(0)=\dim \Pi_{\bfJ}$)
and by \cite{Ha77,V25} it cannot be smaller. Hence it suffices to show that for any
irreducible
component of the special fiber $\Pi_{\bfJ}(0)$ there exists a point satisfying the following properties:
\begin{itemize}
	\item the point belongs to the above closure,
	\item the point belongs to a single irreducible component. 
\end{itemize} 
The irreducible components of $\Pi_{\bfJ}(0)$ are labeled by the cardinality $k$ sets
$S\subset \bZ_n$ such that $\bfJ\le \bfJ(S)$. Let $S$ be such a set and let 
$p(S)\in \Pi_{\bfJ}(0)$ be the corresponding torus fixed point:
\begin{equation}\label{eq:p(S)}
p(S)_b = \mathrm{span}\{w_i^{(b)}:\ i\in J(S)_b\}.
\end{equation}
The point $p(S)$ belongs to a single irreducible component of the juggling variety 
and hence to a single component of $\Pi_{\bfJ}(0)$. Now it suffices to note that 
$p(S)$ (defined by  \eqref{eq:p(S)}) belongs to the fiber $\Pi_{\bfJ}(\veps)$
also for all nonzero values of $\veps$. In fact, $p(S)$ (as a representation of the 
equioriented quiver) is a direct sum of $n$-dimensional subrepresentations 
$\mathrm{span}\{w^{(s)}_n,w^{(s+1)}_{n-1},\dots,w^{(s+n-1)}_1\}$. Each such span is
a subrepresentation with respect to the maps $M(\veps)$ and $p(S)_b\in X^-_{J_b}$
for all $b$ (since $p(S)\in \Pi_{\bfJ}(0)$). We conclude that $p(S)\in \Pi_{\bfJ}(\veps)$
for any $\veps$.
\end{proof}

For a juggling pattern $\bfJ$, let $\emph{Pl}_\bfJ=\bigoplus_{m\in\bZ_{\ge 0}} \emph{Pl}_\bfJ(m)$ 
be the homogeneous
coordinate ring of the classical positroid variety $\Pi_\bfJ$ with respect to the embedding 
$\imath: \Pi_\bfJ\subset \bP(\Lambda^k W)$. 
By projective normality of positroid varieties (\cite{KLS13}, Section 5.5 and \cite{Lam14},
Theorem 9.5) and the cohomology vanishing for projected Richardson varieties 
(\cite{KLS14}, Theorem 4.5) one has
\[
\dim \emph{Pl}_\bfJ(m) = \dim H^0(\Pi_\bfJ,\imath^*\eO(m)), \ 
H^{>0}(\Pi_\bfJ,\imath^*\eO(m))=0
\]
(with $\eO(m)=\eO^{\T m}$).

Now let us consider the homogeneous coordinate ring of the degenerate positroid variety $\Pi_\bfJ(0)$ given by
 ${\emph{Pl}}_\bfJ(0)=\bigoplus_{\bfm\in\bZ^{\bZ_n}} \emph{Pl}_\bfJ(0;\bfm)$
(recall the embedding  
$\jmath: \Pi_\bfJ(0)\subset \prod_{b\in\bZ_n} \bP(\Lambda^k W^{(b)})$). 
We denote by $\eO_b$ the standard Picard generator of the projective 
space $\bP(\Lambda^k W^{(b)})$.
We put forward the following conjecture. 

\begin{conj}\label{conj:higercoh}
For any $\bfm\in\bZ_{\ge 0}^{\bZ_n}$, let $\eO(\bfm)=\bigotimes_{b\in\bZ_n} \eO_b(m_b)$.
Then $\emph{Pl}_\bfJ(0;\bfm)\simeq H^0(\Pi_\bfJ(0),\jmath^*\eO(\bfm))$ and  
the higher   cohomology vanishes $$H^{>0}(\Pi_\bfJ(0),\jmath^*\eO(\bfm))=0.$$
\end{conj}

Assuming Conjecture \ref{conj:higercoh} holds true, Theorem \ref{thm:flat} implies the
equality
$\dim {\emph{Pl}}_\bfJ(0;\bfm) = \dim \emph{Pl}_\bfJ(\sum_b m_b)$.

\appendix
\section{Projective spaces}
In this section we consider the $k=1$ case, i.e. global positroid varieties inside the global projective spaces $\underline{\bP}^{n-1}=\underline{\Gr}_{1,n}$. 
Let $\bfJ=(J_0,\dots,J_{n-1})$ be a $(1,n)$ juggling pattern; in particular, each $J_b$ is an element of $[n]$. 
Let $L(\bfJ)\subset\bZ_n$ be the set of indices $b$ such that $J_b=1$ ($b\in L(\bfJ)$ iff $J_b=1$).
Let $\ell(\bfJ)$ be the cardinality of $L(\bfJ)$.  

\begin{lem}\label{lem:k=1proj}
For any $\veps\ne 0$, the variety $\Pi_\bfJ(\veps)$ is isomorphic to the projective space 
$\bP^{\ell-1}$. For a point $(U_b)_b\in \Pi_\bfJ(\veps)$ a Pl\"ucker coordinate 
$\Delta^{(b)}_i$ vanishes on $U_b$ unless $b+i-1\in L(\bfJ)$.
\end{lem}
\begin{proof}
Let $U_b$ be a component of a point $(U_b)_b\in \Pi_\bfJ(\veps)$. Then 
$$\Delta^{(b+i-1)}_1(U_{b+i-1})=\Delta^{(b)}_i(U_b).$$ If $J_{b+i-1}>1$, then the condition 
$U_{b+i-1}\in X^-_{J_{b+i-1}}$
implies the equality $\Delta^{(b+i-1)}_1(U_{b+i-1})=0$, which proves the desired claim.
\end{proof}

\begin{lem}\label{lem:k=1projdeg}
The special fiber 
$\Pi_\bfJ(0)$ has $\ell(\bfJ)$ irreducible components, each of dimension $\ell(\bfJ)-1$.
For a point $(U_b)_b\in \Pi_\bfJ(0)$ the Pl\"ucker coordinate $\Delta^{(b)}_i$
vanishes on $U_b$ unless $b+i-1\in L(\bfJ)$.	
\end{lem}
\begin{proof}
The second part is proved in the same way as above. To prove the first part	
we recall that the irreducible components of $\Pi_\bfJ(0)$ 
are labeled by the singletons  $S\subset \bZ_n$ such that $\bfJ(S)\ge\bfJ$.
This condition is equivalent to the condition that  $\bfJ(S)_b=1$ implies $J_b=1$ 
(recall that $k=1$ in this section).
Since there exists a unique 
$S$ such that  $\bfJ(S)_b=1$ for a fixed $b$, we arrive at the desired claim.	
\end{proof}

For a nonnegative integer $M$ and a projective space $\bP$, let $\eO(M)$ be the $M$-th tensor
 power of the ample generator $\eO_\bP(1)$ of the Picard group of a projective space.
\begin{cor}
For $M\ge 0$, one has $\dim H^0(\Pi_\bfJ,\eO(M)) = \binom{M+\ell(\bfJ)-1}{M}$. 	
\end{cor}

We provide a flat basis of the homogeneous coordinate ring of the family $\underline{\Pi}_\bfJ$. 
Recall the embedding $\underline\Pi_\bfJ\subset\bA^1\times \prod_{b\in\bZ_n} \bP(W^{(b)})$ and the homogeneous coordinates $\Delta^{(b)}_i$ in the $b$-th factor $\bP(W)$. 
The homogeneous coordinate ring $\underline{\emph{Pl}}_\bfJ$ is a quotient of the 
polynomial ring in variables $\Delta^{(b)}_i$, $b\in\bZ_n$, $i\in [n]$ over $\bC[\veps]$. 
The ring $\underline{\emph{Pl}}_\bfJ$ is multigraded by  
$\bZ_{\ge 0}^{\bZ_n}$, $\underline{\emph{Pl}}_\bfJ=\bigoplus_{\bfm} {\emph{Pl}}_\bfJ(\bfm)$,
where  $\bfm=(m_0,\dots,m_{n-1})$, $m_b\in\bZ_{\ge 0}$. Explicitly, $m_b$ is equal to the 
number of variables of the form $\Delta^{(b)}_\bullet$ in a monomial.

Our goal is to present a basis of the ring $\underline{\emph{Pl}}_\bfJ$ 
which is compatible with the decomposition into
the direct sum of ${\emph{Pl}}_\bfJ(\bfm)$. Each element of the basis is of the form
\begin{equation}\label{eq:admmon}
\prod_{b\in\bZ_n} \Delta^{(b)}_{i^{(b)}_1}\dots \Delta^{(b)}_{i^{(b)}_{m_b}},\ i^{(b)}_1\le\dots\le i^{(b)}_{m_b}.
\end{equation}
The collection $\bfm=(m_b)_b$ is called the multidegree of the monomial \eqref{eq:admmon}.

\begin{dfn}\label{dfn:admmon}
A monomial \eqref{eq:admmon} is called $\bfJ$-admissible if 
\begin{itemize}
\item $b+i^{(b)}_r-1\in L(\bfJ)$ for all $b\in\bZ_n$, $r\in [m_b]$,
\item if $i^{(b)}_u > s$ for some $b\in \bZ_n$, $u\in[m_b]$, $s\in [n-1]$, then 
$i^{(b+s) \!\!\mod n}_r\le i^{(b)}_u - s$ for all $r$.
\end{itemize}
\end{dfn}

\begin{rem}
Let us comment on the first condition $J_{b+i^{(b)}_r-1}=1$ (here as usual the lower index of 
$J$ is understood as an element of $\bZ_n$).
The condition implies that $i^{(b)}_r\ge J_b$, i.e.  cuts out the corresponding Schubert variety,
because if $i^{(b)}_r < J_b$, then
$J_{b+i^{(b)}_r-1} = J_b - i^{(b)}_r + 1 > 1$. Moreover, if all 
numbers $m_b$ are strictly positive, then the conditions $J_{b+i^{(b)}_r-1}=1$
are equivalent to the Schubert conditions $i^{(b)}_r\ge J_b$ (taking into account the second condition in Definition \ref{dfn:admmon}). In fact,
assume that all the Schubert conditions $i^{(b)}_r\ge J_b$ hold and all
the conditions of the second type from Definition \ref{dfn:admmon} hold as well.
Then  
\[
1 = i^{(b)}_r - (i^{(b)}_r - 1) \ge i^{(b+i^{(b)}_r-1)}_1 \ge J_{b+i^{(b)}_r-1}. 
\] 
However, if some of the multiplicities $m_b$ vanish, then the conditions 
$i^{(b)}_r\ge J_b$ are weaker than the conditions of the first type in 
Definition \ref{dfn:admmon}.
\end{rem}

\begin{prop}\label{prop:number}
Let $|\bfm|=\sum_{b\in\bZ_n} m_b$. Then 
the number of $\bfJ$-admissible monomials of multidegree $\bfm$ is equal to 
$\binom{|\bfm|+\ell(\bfJ)-1}{|\bfm|}$.
\end{prop}
\begin{proof}
We first assume that $\ell(\bfJ)=n$, i.e. $J_b=1$ for all $b$, and then deduce the general statement.

The number  $\binom{|\bfm|+n-1}{|\bfm|}$ is equal to the number of $n$-tuples 
$(s_1,\dots,s_n)$ of nonnegative integers such that $s_1+\dots +s_n = |\bfm|$. We establish a weight-preserving bijection between such $n$-tuples and the set of 
admissible monomials of the form \eqref{eq:admmon}. 
Explicitly, weight preserving means that for all $j=1,\dots,n$
$$\#\{(r,b):\ i_r^{(b)}+b = j \mod n\} = s_j.$$   
In words, we are counting the number of Pl\"ucker variables $\Delta_{i_r}^{(b)}$ -- 
factors of a monomial -- such that $i_r+b$ is fixed modulo $n$
(this is compatible with the decomposition into indecomposable summands of the representation
$M(\veps)$, $\veps\ne 0$  of the equioriented cyclic quiver, see Section \ref{sec:qG}).

So our task is to distribute $s_j$ balls into the cells labeled by a vertex $b$ 
and a position $i$ with $i+b=j$  taking care of the admissibility condition. Let 
$$q_i^{(b)}=\#\{(r,b):\ i_r^{(b)}=i\}$$ be the number of factors of the form $\Delta_i^{(b)}$
in \eqref{eq:admmon}.  
We first observe that if  $q^{(b)}_{r-b}>0$ with $r-b\ne 1$, then $q^{(r-1)}_{1}=m_{r-1}$
(i.e. all the balls in the $(r-1)$-st column are concentrated at the first position).
In fact, since we are constructing an admissible monomial, 
$q^{(b)}_{r-b}>0$, $r-b>1$ implies $q^{(r-1)}_{>1}=0$. So we conclude that in the process of distributing
our balls we are forced to start with the cells corresponding to the lowest possible position
for each vertex (i.e. with $q^{(\bullet)}_{1}$). There are two options here: either 
$s_r\le m_{r-1}$ and all the $s_r$ balls produce $q^{(r-1)}_{1}=s_r$, or $s_r>m_{r-1}$ and
we are left to distribute further balls among the cells. So after the first step all 
the entries $q^{(\bullet)}_1$ are already uniquely determined.

Assume that for some values of $r$ the second option as above is realized (i.e. 
$s_r>m_{r-1}$). Let us look at the value of $q^{(r-2)}_{2}$.
If  $q^{(b)}_{r-b}>0$, $r-b>2$, then $q^{(r-2)}_{>2}=0$ (as a consequence of admissibility).
Therefore, there are two options here: either $q^{(r-2)}_1=m_{r-2}$ or $q^{(r-2)}_1 < m_{r-2}$.
In the first case all the values $q^{(r-2)}_{>1}$ (including $q^{(r-2)}_2$) automatically vanish.
In the second case, we fill $q^{(r-2)}_2$ with the balls from the initial group of $s_r$ 
(left from filling $q^{(r-1)}_1$) until either all the $s_r$ balls are used or the total number
of $m_{r-2}$ balls at this vertex is achieved. After this step all the values of $q^{(\bullet)}_{2}$
are fixed. We then proceed with $q^{(\bullet)}_{3}$, etc.

We are left to consider the case of a general positroid variety. The first condition of Definition
\ref{dfn:admmon} says that in order to pass from the maximal case considered above 
($J_b=1$ for all $b\in\bZ_n$) to the arbitrary $\bfJ$ case one needs to keep only the entries
$q^{(b)}_r$ such that $b+r-1\in L(\bfJ)$ (to be compared with Lemma \ref{lem:k=1proj}).
To prove our proposition in this case we just repeat the argument above for the relevant entries.
\end{proof}

Recall the global Pl\"ucker algebra $\underline{\emph{Pl}}_\bfJ$ which is the quotient of the ring $\bC[\veps,\Delta^{(b)}_I]_{b,I}$ by the ideal of relations vanishing on $\underline\Pi_\bfJ$.

\begin{lem}\label{lem:span}
$\bfJ$-admissible monomials span the ring $\underline{\emph{Pl}}_\bfJ$ as a 
$\bC[\veps]$ module.
\end{lem}
\begin{proof}
The following elements vanish on $\underline{\Pi}_\bfJ$:
\begin{equation}\label{eq:relk=1}
\begin{tabular}{c c c l}
 &  $\Delta^{(b)}_i$, &  & $b+i-1\notin L(\bfJ)$,\\
$\Delta^{(b)}_i \Delta^{(b+s)}_j$ & - & 
$\Delta^{(b)}_{j+s} \Delta^{(b+s)}_{i-s}$, & $i-s\ge 1,\ j+s\le n$,\\
$\veps \Delta^{(b)}_i \Delta^{(b+s)}_j$ & - & 
$\Delta^{(b)}_{j+s} \Delta^{(b+s)}_{i-s+n}$, & $i-s\le 0,\ j+s\le n$,\\
$\Delta^{(b)}_i \Delta^{(b+s)}_j$  & - &
$\veps \Delta^{(b)}_{j+s-n} \Delta^{(b+s)}_{i-s}$, & $i-s\ge 0,\ j+s> n$,\\
$\Delta^{(b)}_i \Delta^{(b+s)}_j$ & -  &
$\Delta^{(b)}_{j+s-n} \Delta^{(b+s)}_{i-s+n}$, & $i-s\le 0,\ j+s> n$.
\end{tabular} 	
\end{equation}
Now assume we are given a nonadmissible monomial, i.e. it contains a binomial of the form
$\Delta^{(b)}_i \Delta^{(b+s)}_j$ with $i\ge s+1$, $j>i-s$ (if the first condition in Definition \ref{dfn:admmon} is violated, then 
the corresponding monomial vanishes on the global positroid variety thanks to Lemma \ref{lem:k=1proj}).
We consider two cases: $j>n-s$ and $j\le n-s$. 
In the first case one has
\[
\Delta^{(b)}_i \Delta^{(b+s)}_j = 
\veps \Delta^{(b)}_{j+s-n} \Delta^{(b+s)}_{i-s} 
\]
and in the second case ($i-s<j\le n-s$) one has
\[
\Delta^{(b)}_i \Delta^{(b+s)}_j = 
\Delta^{(b)}_{j+s} \Delta^{(b+s)}_{i-s}. 
\]
In both cases the product of lower indices of variables becomes smaller (i.e. $ij > (i-s)(j+s-n)$ in the first case and $ij>(i-s)(j+s)$ in the second). Hence we obtain the claim of our Lemma.
\end{proof}

\begin{cor}\label{cor:redschk=1}
For any $\veps\ne 0$, quadratic relations \eqref{eq:relk=1} define the reduced scheme structure.
The admissible monomials form a basis of the coordinate ring of the corresponding fiber.
\end{cor}
\begin{proof}
By Lemma \ref{lem:span} admissible monomials span the quotient ring by relations \eqref{eq:relk=1}.
By Proposition \ref{prop:number} the number of admissible monomials is equal to the binomial coefficient,
which is equal to the dimension of the corresponding component of the homogeneous coordinate
ring of the projective space.     
\end{proof}

\begin{rem}
One can also prove Corollary \ref{cor:redschk=1} by showing that $\bfJ$-admissible 
monomials are linearly independent in ${\emph{Pl}}_\bfJ(\veps)$ for $\veps\ne 0$,
which is equivalent to the statement that a nontrivial homogeneous linear combination
of $\bfJ$-admissible monomials cannot vanish on $\Pi_\bfJ(\veps)$.  
\end{rem}

\begin{thm}\label{thm:freebas}
$\bfJ$-admissible monomials form a free basis of $\underline{\emph{Pl}}_\bfJ$ as a
$\bC[\veps]$ module. 
\end{thm}
\begin{proof}
Thanks to Corollary \ref{cor:redschk=1} above it remains to treat the case $\veps=0$. 
More precisely, let us show 
that the admissible monomials are linearly independent in the homogeneous coordinate ring of 
a positroid variety $\Pi_\bfJ(0)$. We work out the  maximal case $L(\bfJ)=\bZ_n$ (i.e. 
$J_b=1$ for all $b$) with $\Pi_\bfJ(0)$ being the juggling variety. By
Lemma \ref{lem:k=1projdeg} and Definition \ref{dfn:admmon}, the case of general $\bfJ$ is 
treated in the same way.

Recall that in the maximal case $\Pi_\bfJ(0)$ is a union of $n$ irreducible components labeled by the elements $a\in\bZ_n$ with 
the $a$-th component consisting of collections $(U_b)_b\in \Pi_\bfJ(0)$ such that
\[
U_{a+1}\subset \mathrm{span}\{w_i\}_{i=1}^{n-1},\ 
U_{a+2}\subset \mathrm{span}\{w_i\}_{i=1}^{n-2},\dots,
U_{a-1} =\mathrm{span}\{w_1\}.
\]
Let us consider a nontrivial linear combination $C$ of admissible monomials. We pick a monomial
in $C$
\begin{equation}\label{eq:monchoice}
\prod_{b\in\bZ_n} \Delta^{(b)}_{\ell^{(b)}_1}\dots \Delta^{(b)}_{\ell^{(b)}_{m_b}},\ \ell^{(b)}_1\le\dots\le \ell^{(b)}_{m_b}
\end{equation}
satisfying the following condition:
let  $\ell^{(b_0)}_{r_0}$ be a maximal index in the product above.
Then all the indices of all the Pl\"ucker variables in all the summands of our 
linear combination $C$ are no larger than   $\ell^{(b_0)}_{r_0}$.
In particular, since monomial \eqref{eq:monchoice} is admissible 
we derive that 
\begin{equation}\label{eq:cond}
\ell^{(b_0+r_0-1)}_{>1} =0, \ \ell^{(b_0+r_0-2)}_{>2} =0, \dots,
\ell^{(b_0+r_0+1)}_{n} =0.
\end{equation}
Hence, the monomial \eqref{eq:monchoice} defines a nontrivial
function on the component corresponding to vertex $b_0+r_0$. In order
to prove that our nontrivial combination $C$ cannot vanish on the degenerate positroid variety it suffices to prove that it does 
not vanish on the $(b_0+r_0)$-th irreducible component. However,
the proof of Proposition \ref{prop:number} implies that two different admissible monomials 
cannot have the same weight (the weight is the $n$-vector whose $j$-th component 
is equal to the number of Pl\"ucker variables $\Delta^{(b)}_r$ showing up in a monomial 
such that $b+r=j$ modulo $n$). In our situation we have even more restrictions: 
we only care about monomials satisfying conditions \eqref{eq:cond}, i.e. monomials not vanishing on our irreducible component.
We conclude that $C$ does not vanish completely on $\Pi_\bfJ(0)$.
\end{proof}

\begin{cor}\label{cor:conjk=1}
 Conjecture \ref{conj:reduced} holds true for $k=1$.
\end{cor}
\begin{proof}
Let us show that relations \eqref{eq:relk=1} belong to the saturation of the ideal
generated by  ${\mathcal{I}}^{(a,b)}_{1,n}(\veps)$, $a,b\in\bZ_n$ and linear forms $\Delta^{(a)}_I$, $I\not\succeq J_a$, $a\in\bZ_n$. We denote this ideal by $\mathcal{I}'(\veps)$.

First, let $\veps\ne 0$. For a line  $U_b\in \bP(W^{(b)})$, one has
\[
\Delta^{(b)}_i (U_b) = 
\begin{cases}
\Delta^{(b+s)}_{i-s} (U_{b+s}), & i>s\\
\veps^{-1} 	\Delta^{(b+s)}_{i-s+n} (U_{b+s}), & i\le s.
\end{cases}
\]
Hence, the last four lines are exactly the generators of ${\mathcal{I}}^{(b,b+s)}_{1,n}(\veps)$. 

Now let us show that variables $\Delta^{(b)}_i$ belong to the saturated ideal for $b+i-1\notin L(\bfJ)$. If $b+i-1\notin L(\bfJ)$, then $J_{b+i-1}>1$ and hence 
$\Delta^{(b+i-1)}_1\in\mathcal{I}'(\veps)$. One has the following relations in 
${\mathcal{I}}^{(b,b+i-1)}_{1,n}(\veps)$ for any $n\ge t\ge J_{b+i-1}$:
\begin{equation}\label{eq:satur}
\Delta^{(b)}_i \Delta^{(b+i-1)}_t =
\begin{cases}
\Delta^{(b)}_{t+i-1} \Delta^{(b+i-1)}_1, & t+i-1\le n\\
\veps \Delta^{(b)}_{t+i-1-n} \Delta^{(b+i-1)}_1, & t+i-1> n.	
\end{cases}
\end{equation}
We conclude that $\Delta^{(b)}_i\Delta^{(b+i-1)}_t\in {\mathcal{I}}^{(b,b+i-1)}_{1,n}(\veps)$ for
any $t$ and hence $\Delta^{(b)}_i$ belongs to the saturated ideal.

Now let us consider the case $\veps=0$. 
Quadratic relations in \eqref{eq:relk=1} with $\veps=0$ vanish on $\Pi_\bfJ(0)$ (they 
are polynomials in $\veps$ and they vanish on $\Pi_\bfJ(\veps)$ for nonzero $\veps$) and one immediately see 
all quadratic relations belong to ${\mathcal{I}}^{(b,b+s)}_{1,n}(0)$.  
As for variables $\Delta^{(b)}_i$ , 
$b+i-1\notin L(\bfJ)$, one needs to repeat the proof given above for $\veps\ne 0$. 
The only difference is that  in \eqref{eq:satur} we need to substitute $\veps=0$.

We conclude that for any $\veps$ all the elements from \eqref{eq:relk=1} vanish on 
$\Pi_\bfJ(\veps)$. We also know from Lemma \ref{lem:span} and Theorem \ref{thm:freebas}
that $\bfJ$-admissible monomials span the quotient modulo the ideal generated by 
\eqref{eq:relk=1} and that they are linearly independent in the Pl\"ucker algebra. 
We conclude that \eqref{eq:relk=1} indeed gives a reduced scheme structure on $\Pi_\bfJ(\veps)$. 
\end{proof}

We close with a proof of Conjecture \ref{conj:comb} in the $k=1$ case.

\begin{prop}
Conjecture \ref{conj:comb} holds true for $k=1$ and any $n$.
\end{prop}
\begin{proof}
We consider the $b=0$ case of the conjecture, the general case is proved along the same lines.

There is only one irreducible component of $\Pi_\bfJ(0)$ which can surject onto $\bP_\bfJ$, since the image
of all components but one is contained in a projective space of dimension smaller than $\ell(\bfJ)$. So it suffices to prove the surjectivity of the map $\Pi_\bfJ(0)\to \bP_\bfJ$. 

Let $U_0$ be a point in $\bP_\bfJ$. Our goal is to lift it to a point $(U_0,\dots,U_{n-1})\in\Pi_\bfJ(0)$. Recall the set $L(\bfJ)$ consisting of $b\in\bZ_n$ such that 
$J_b=1$. 
Let $k_1\in \{0,\dots,n-1\}$ be the largest number such that 
$M_{0\to k_1}(0)U_0\ne 0$.
Let $k_2\in \{0,\dots,n-1\}$ be the largest number such that 
$-k_2\notin L(\bfJ)$.
We define $U_b$ as follows:
\[
U_b = 
\begin{cases}
M_{0\to b}(0)U_0, & 0\le b\le k_1,\\
\varphi^{-b}U_0, & -k_2\le b \le 0,\\	
\mathrm{span}\{w_{J_b}\}, &  \text{ otherwise}.	 
\end{cases}	
\]
First, all the spaces are obviously one-dimensional. Second, $M_{b\to b+1}(0)U_b\subset U_{b+1}$. In fact, this is obvious for $b$ from the first and the third regions. Now let us look at $U_{-1},\dots,U_{-k_2}$. Since $U_0\in \Pi_\bfJ$, $\varphi^{-1}U_0\in X^-_{J_{-1}}$ 
(the coefficient of $w^{(-1)}_1$ in $\varphi^{-1}U_0$ vanishes). If $J_{-1}\ne 1$, this implies 
$M_{-1\to 0} \varphi^{-1}U_0 = U_0$. Similarly, 
for any $0\le b\le k_2$, one gets
$M_{-b-1\to -b} \varphi^{-1}U_{-b} = U_{-b}$.
Finally, $U_b\in X^-_{J_b}$ for all $b$.
\end{proof}

\end{document}